\documentclass[12pt]{extarticle}

\usepackage[utf8]{inputenc}
\usepackage{helvet}

\usepackage[margin=1in]{geometry}
\usepackage{graphicx}
\usepackage{titlesec}
\usepackage{abstract}
\usepackage{parskip}
\usepackage{amsmath,amssymb,amsthm}
\usepackage{float}
\usepackage[algo2e,ruled,vlined]{algorithm2e}
\usepackage{tikz}
\usepackage{pgf}
\usepackage{tabularx}
\usepackage{array}
\usepackage{bm}
\usepackage{wasysym}
\usepackage[all]{xy}
\usepackage[english,activeacute]{babel}

\newtheorem{Theorem}{Theorem}
\newtheorem{Lemma}{Lemma}
\newtheorem{Corollary}{Corollary}
\newtheorem{Proposition}{Proposition}

\newcounter{rownumber}
\newcommand{\rownumber}{\stepcounter{rownumber}\arabic{rownumber}}

\titleformat{\section}{\large\bfseries}{\thesection.}{0.5em}{}
\titleformat{\subsection}{\normalsize\bfseries}{\thesubsection.}{0.5em}{}

\title{\textbf{Factorizations in Geometric Lattices}}
\author{
  Alex Aguila\thanks{Miami Dade College, Padron Campus, Mathematics Department} \and 
  Elvis Cabrera \and 
  Jyrko Correa-Morris\thanks{Corresponding author: \texttt{jcorrea7@mdc.edu}}
}
\date{}

\begin{document}

\maketitle

\vspace{-1em}
\begin{center}
\textbf{\large Abstract}
\end{center}
\vspace{0.75em}  
\begin{center}
\begin{minipage}{0.85\linewidth}
\small
This article investigates atomic decompositions in geometric lattices isomorphic to the partition lattice $\Pi(X)$ of a finite set $X$, a fundamental structure in lattice theory and combinatorics. We explore the role of atomicity in these lattices, building on concepts introduced by D.D. Anderson, D.F. Anderson, and M. Zafrullah within the context of factorization theory in commutative algebra. As part of the study, we first examine the main characteristics of the function $\mathfrak{N}\colon \Pi(X) \rightarrow \mathbb{N}$, which assigns to each partition $\pi$ the number of minimal atomic decompositions of $\pi$. We then consider a distinguished subset of atoms,  $\mathcal{R}$, referred to as the set of red atoms, and derive a recursive formula for $\pmb{\pi}(X, j, s, \mathcal{R})$, which enumerates the rank-$j$ partitions expressible as the join of exactly $s$ red atoms.
\end{minipage}
\end{center}
\vspace{1em}


	\section{Introduction} 
	
	The concept of atomic decomposition has been of great interest in both philosophy and physics for centuries, whose genesis presumably goes back to Leucippus and Democritus, 400 BC. During the last decades, different formal notions of atomicity have been formalized and studied in several areas of mathematics. The question of the existence of irredundant atomic decompositions in lattices with certain additional structures has been investigated since mid-past century. In this regard, P. Crawley introduced a decomposition theory for nonsemimodular lattices \cite{PCrawley1} and collaborated with R.P. Dilworth in the development of a decomposition theory for lattices without chain conditions \cite{RPDilworth1}. These contributions found further development in the work of M. Enre \cite{MErne1}, and, more recently, in the contributions of T. Albu \cite{TAlbu1}, which are inspired in applications to Grothendieck categories and torsion theories. A compilation of Dilworth's theorems is presented in \cite{RPDilworth2}. In the context of commutative monoids and integral domains  atomic decompositions, also referred to as factorizations, were considered in the sixties by L. Carlitz [2] and P. Cohn [3], in the seventies by A. Grams [6] and A. Zaks [10], and in the eighties by J. L. Steffan [7] and R. Valenza [8]. In 1990, D. D. Anderson, D. F. Anderson, and M. Zafrullah provided the first systematic investigation of atomic decompositions in the context of integral domains [1]. From that point on, atomic decompositions have been systematically studied by many authors under the umbrella of factorization theory (see [4, 5] and references therein).
	The reach of atomic decompositions has expanded significantly, influencing core developments in mathematics and computer science. In this regard, there are several open problems in algebra about the atomicity of certain classes of monoids  \cite{Atomicity1,Atomicity2,Atomicity3}. In Probability, Mishura et. al. investigated the atomic decompositions and inequalities for vector-valued discrete-time martingales \cite{YSMishura}. Along similar lines, X. Zhang investigated the atomic decompositions of Banach lattice-valued martingales and used them to study the relation of the martingale spaces \cite{XZhang1}. In the domain of ontologies, \cite{ChVescovo1} \cite{MHorridge1} used atomic decompositions to accurately represent the family of all locality-based modules of a given ontology.

	This paper delves into atomic decompositions within the context of geometric lattices. To the best of our knowledge, no existing work comprehensively addresses this topic from the perspective presented here. When analyzing finite geometries, two key aspects must be considered: \emph{incidence}, which deals with inclusion relationships between geometric objects (e.g., ``a point lies on a line"), and \emph{measurement}, which pertains to properties such as length, area, and angles. By focusing solely on incidence, we can establish fundamental properties that a finite geometry should satisfy: 1. There exists a base set of points, say \(\Omega\), and the geometric elements (e.g., points, lines, planes) are subsets of this base set. 2. Every point in \(\Omega\) is an element of the geometry, and \(\Omega\) itself is also considered an element of the geometry. 3. The intersection of any geometric elements is also an element of the geometry. 4. A rank or dimension function exists, which captures the hierarchical structure of the geometric elements. If we order these geometric elements by set inclusion, the resulting structure forms a lattice, where the points of the geometry correspond to the atoms of this lattice. This motivates the definition of a \emph{finite geometric lattice} as an algebraic lattice $(L, \preceq, \vee,\wedge)$, where $\preceq$ is the partial order relation, and $\vee$ and $\wedge$ denote the join operator and the meet operator, respectively, which is \emph{semimodular} (i.e., $a\wedge b \sqsubset a$ implies $b\sqsubset a\vee b$, where $\sqsubset$ stands for the covering relation (i.e., $c \sqsubset d$ if and only if $c\preceq d$ and for all $y\in L$, $c\preceq y\preceq d$ implies that either $y=c$ or $y=d$)) and \emph{atomistic} (i.e., every element different from the minimum element of the lattice can be expressed as the join of atoms (i.e., join irreducible elements of $L$)). To highlight the significance of geometric lattices in both algebraic and geometric contexts, it is worth noting two well-established facts: a geometric lattice is isomorphic to (1) the lattice of ideals of a finite, atomistic, semimodular, and principally chain lattice; and (2) the lattice of flats in a simple matroid.
	
	A \emph{finite matroid} is a pair $M=(S,\mathcal{I})$ where $S$ is a finite set and $\mathcal{I}\neq\emptyset$ is a subset of $2^S:=\{H\,:\, H\subseteq S\}$ such that (1) if $J\in\mathcal{I}$ and $J'\subseteq J$, then $J'\in\mathcal{I}$; and (2) for every $H\subseteq S$, the maximal elements of $2^{H}\cap \mathcal{I}$ have the same cardinality. In a matroid $M$, the rank of a subset $H\subseteq S$, denoted by $\text{rk}(H)$, is defined to be the maximum cardinality of those elements of $\mathcal{I}$ that are contained in $H$, i.e., $\text{rk}(H)=\max\{|J|\,:\, J\in \mathcal{I}\}$. A $k$-flat of $M$ is a maximal subset of rank $k$. The set of all flats of $M$ ordered by inclusion is known to be geometric lattice, which we shall denote by $L(M)$. $M$ is said to be simple if there is no $s\in S$ satisfying $\text{rk}(\{s\})=0$.
	
	\begin{Theorem} (\cite{StanleyHA}) Let $L$ be a finite lattice. $L$ is geometric if and only if $L$ is isomorphic to $L(M)$ for some simple matroid $M$.
	\end{Theorem}
	
	An archetypal example of a geometric lattice is the set of all partitions of a finite set \(X\) with \(n\) elements, ordered by partition refinement. 
	
	A partition $\pi$ of a finite set $X$ with $n\geq 1$ elements is a collection of disjoint nonempty subsets of $X$ whose union is $X$. That is, $\pi = \{ \textsc{b}_{1}, \textsc{b}_{2},..., \textsc{b}_{k}\}$, with $\textsc{b}_{i} \neq \emptyset$, $\textsc{b}_{i} \cap \textsc{b}_{j} = \emptyset$ for $i \neq j$, and $\underset{i = 1}{\overset{k}{\cup}} \textsc{b}_{i} = X$. The set of all partitions of the set $X$ is endowed with the refinement of partitions. The partition $\pi$ is said to refine the partition $\pi'$ if and only if every block of $\pi'$ is the union of some blocks of $\pi$ (in particular, every block of $\pi$ is contained in some block of $\pi'$). The notation $\pi\preceq\pi'$ means $\pi$ refines $\pi'$, which is also read as $\pi$ is finer than $\pi'$, and as $\pi'$ is coarser than $\pi$. If $\pi\preceq\pi'$, and for every partition $\hat{\pi}$, $\pi\preceq\hat{\pi}\preceq \pi'$ implies that either $\hat{\pi}=\pi$ or $\hat{\pi}=\pi'$, then $\pi'$ is said to cover the partition $\pi$. Henceforth, we denote by $\Pi(X)$ the space of all partitions of $X$. We shall also use the notation $x\pi x'$ to mean that the elements $x,x'\in X$ lie in the block of the partition $\pi\in\Pi(X)$.
	
	The operations with partitions are another important ingredient of $\Pi(X)$, which endow this set with a rich structure. For any two partitions $\pi$ and $\pi'$ in $\Pi(X)$, there can always be found partitions that refine both. The coarsest partition that satisfies this property is denoted by $\pi\wedge \pi'$ and is called the \emph{meet} of $\pi$ and $\pi'$. The blocks of $\pi\wedge\pi'$ are all possible nonempty intersections of a block of $\pi$ and a block of $\pi'$. Similarly, 
	for any two partitions $\pi$ and $\pi'$ there can always be found partitions that are refined by both. The finest partition that satisfies this property is denoted by $\pi\vee \pi'$ and is called the \emph{join} of $\pi$ and $\pi'$. Two elements $x,x'\in X$ are placed in the same block of $\pi\vee\pi'$ if and only if there is a sequence $x=x_{i_1}, x_{i_2}, \ldots, x_{i_k}=x'$ such that, for all $j\in\{1,2,\ldots,k-1\}$, $x_{i_j}$ and $x_{i_{j+1}}$ are either placed in the same block of $\pi$ or in the same block of $\pi'$. Thus, the blocks of $\pi\vee\pi'$ are the subsets of $X$ that can be expressed both as unions of blocks of $\pi$ and as unions of blocks of $\pi'$, and are minimal with respect to this property. These operations convert $\Pi(X)$ into a lattice. Moreover, it is well-known that for any two partitions $\pi$ and $\pi'$, if $\pi$ covers $\pi\wedge\pi'$, then $\pi\vee\pi'$ covers $\pi'$, which makes $\Pi(X)$ an upper semimodular lattice.
	
	  There are two distinguished elements in the lattice \( \Pi(X) \): the finest partition, denoted by \( \mathrm{m}_X \), in which all blocks are singletons, and the coarsest partition, denoted by \( \mathrm{g}_X \), consisting of a single block equal to \( X \). These serve as the neutral elements for the meet and join operations, respectively.
	 
	 Moreover, \( \Pi(X) \) is a \emph{ranked lattice}, with rank function defined by
	 \(\operatorname{rank}(\pi) = n - |\pi|\),
	 where \( |\pi| \) is the number of blocks in the partition \( \pi \), and \( n = |X| \). This rank function measures the distance from the finest partition, \( \mathrm{m}_X \), and organizes \( \Pi(X) \) into layers based on the number of blocks.
	 
	 Each covering relation corresponds to merging exactly two blocks, decreasing the number of blocks by one and increasing the rank by one. Every maximal chain has length \( n - 1 \) and the rank increases uniformly along chains.
	
 Furthermore, the partitions of $X$ which cannot be but trivially decomposed as the join of partitions of $X$ are called the \emph{atoms} of $\Pi(X)$. A partition $\pi$ is an atom if all its blocks are singleton except for exactly one, which contains exactly two elements. From now on, we shall denote by $\pi_{xx'}$ the atom  whose only non-singleton block is $\{x, x'\}$,  by $\mathcal{A}$ the set of all atoms of $\Pi(X)$, and by $\mathcal{A}(\pi)$ the set of those atoms of $\Pi(X)$ that refine the partition $\pi$. It is straightforward to note that every partition $\pi$ is the join of all those atoms in $\mathcal{A}(\pi)$. Thus, $\Pi(X)$ is also atomistic, and therefore a geometric lattice.
	
	Several characterizations have been provided for partition lattices, including those by Ore \cite{Ore} and Sasaki and Fujiwara \cite{Sasaki}. Here, we have chosen to include a characterization that somehow encompasses the previous ones: 
	
	\begin{Theorem} (J. R. Stonesifer and K. P. Bogart \cite{Stonesifer})\label{JRSKPB} Let $L$ be a geometric lattice with a modular copoint $m$ such that, for every point $p$, the interval $[p,1]$ is isomorphic to $\Pi_{n-1}$. Then, for $n\geq5$, $L$ is isomorphic to $\Pi(X)$.
	\end{Theorem}
	
	In the scope of Theorem \ref{JRSKPB}, a copoint should be interpreted as an element $m\in L$ whose rank is one less than the rank of $L$. Additionally, $m$ is said to be a modular element if, for every $x\in L$, the condition $y\preceq x$ implies $y\vee(m\wedge x)= (y\vee m)\wedge x$, for every $y\in L$. Also, 1 stands for the maximum element of $L$. Here is an interesting corollary of Theorem \ref{JRSKPB} for supersolvable geometric lattices ---there exists a maximal chain $\Delta$ of $L$ such that, for every chain $K$ of $L$, the sublattice generated by $K$ and $\Delta$ is distributive.
	
	\begin{Corollary}
		If $L$ is a supersolvable geometric lattice such that every interval $[x,1]$ has a characteristic polynomial which equals the characteristic polynomial of a partition lattice, then $L$ is isomorphic to a partition lattice.
	\end{Corollary}
	
	Young-jin Yoon rephrased Theorem \ref{JRSKPB} for combinatorial geometries \cite{Yoon}, thereby emphasizing the geometric relevance of this result. Given a geometry $G$, a subset $T$ of the points of $G$, $G/T$ stands for the contraction of $G$ by $T$ ---the geometry induced by the geometric lattice $[\text{cl}(T), 1]$ on the set of all flats in $G$ covering the closure of $T$, $\text{cl}(T)$. Moreover, for a geometry $G$, $L(G)$   denotes the geometric lattice all whose elements are the flats of $G$ with the set inclusion as the partial order.
	
	\begin{Theorem} (YJ. Yoon \cite{Yoon})
		If a geometry $G$ has a modular copoint and, for every point $p\in G$, $L(G/p)$ is isomorphic to $\Pi(X)$, then $L(G)$ is isomorphic to $L(G)$ for $n\geq4$.
	\end{Theorem}
	
	This paper focuses on geometric lattices that are isomorphic to a lattice of partitions, including the lattices of all equivalence relations on a finite set and the lattices of all factorizations of an integer $q$, ordered by inclusion. However, this work excludes other important examples of geometric lattices, such as the lattice of all subspaces of a vector space over a finite field.  
	
	An atomic decomposition of a partition $\pi$ is a subset $\mathfrak{a}_{\pi}$ of $\mathcal{A}(\pi)$ such that the join of the atoms in $\mathfrak{a}_{\pi}$ equals $\pi$. While $\mathcal{A}(\pi)$ itself forms an atomic decomposition of the partition $\pi$, it typically includes superfluous atoms. The objective is often to generate a given partition $\pi$ by using the fewest possible atoms. An atom $\pi_{xx'}$ in an atomic decomposition $\mathfrak{a}_{\pi}$ is considered redundant if $\mathfrak{a}_{\pi}-\{\pi_{xx'}\}$ remains an atomic decomposition of $\pi$. Thus, a decomposition is minimal if it contains no redundant atoms. If an atomic decomposition $\mathfrak{a}_{\pi}$ of the partition $\pi$ is not minimal, there exists an atom $\pi_{xx'}$ such that removing it, i.e., redefining $\mathfrak{a}_{\pi}:=\mathfrak{a}_{\pi}-\{\pi_{xx'}\}$,  still results in an atomic decomposition of $\pi$. This process of eliminating superfluous atoms can be repeated until a minimal decomposition is achieved. Since $\mathfrak{a}_{\pi}$ is finite, removing all redundant atoms requires only a finite number of steps. Consequently, every partition $\pi$ admits a minimal atomic decomposition. Henceforth, $M(\pi)$ denotes the set all whose elements are the minimal atomic decompositions of $\pi$.
	
	Let us illustrate this notions with a simple example. The partition $\pi=\{\{x_1,x_2,x_3,x_4\},\{x_5,x_6,x_7\}\}$ of $X=\{x_1,x_2,x_3,x_4,x_5,x_6,x_7\}$ admits the following atomic decompositions (among others):
	
	\begin{eqnarray}\label{decomp1}
	\pi&=& \pi_{x_1x_2}\vee \pi_{x_2x_3}\vee \pi_{x_3x_4}\vee \pi_{x_5x_6}\vee \pi_{x_6x_7},\nonumber\\
	\mathfrak{a}_{\pi} &=& \{\pi_{x_1x_2}, \pi_{x_2x_3}, \pi_{x_3x_4}, \pi_{x_5x_6}, \pi_{x_6x_7}\};
	\end{eqnarray}
		\begin{eqnarray}\label{decomp2}
	\pi&= &\pi_{x_1x_2}\vee \pi_{x_1x_3}\vee \pi_{x_1x_4}\vee \pi_{x_5x_6}\vee \pi_{x_6x_7},\nonumber\\
	\mathfrak{a}'_{\pi} &= &\{\pi_{x_1x_2}, \pi_{x_1x_3}, \pi_{x_1x_4}, \pi_{x_5x_6}, \pi_{x_6x_7}\};
	\end{eqnarray}
	and
	\begin{eqnarray}\label{decomp3}
	\pi&=& \pi_{x_1x_2}\vee \pi_{x_2x_3}\vee \pi_{x_3x_4}\vee \pi_{x_2x_4}\vee \pi_{x_5x_6}\vee \pi_{x_6x_7}\vee \pi_{x_5x_7},\nonumber\\
	\mathfrak{a}''_{\pi} &=& \{\pi_{x_1x_2}, \pi_{x_2x_3}, \pi_{x_3x_4},
	\pi_{x_2x_4},
	\pi_{x_5x_6}, \pi_{x_6x_7},
	\pi_{x_5x_7}\}.
	\end{eqnarray}
	
		\begin{figure}[h!]
		\centering
		\tikzset{every picture/.style={line width=0.75pt}} 
		\vspace{0.5cm}
		\hspace{-1.9cm}
		\begin{tikzpicture}[x=0.75pt,y=0.75pt,yscale=-1,xscale=1]
		
		\draw   (82.9,73.92) .. controls (82.9,60.18) and (94.04,49.03) .. (107.79,49.03) -- (548.94,49.03) .. controls (562.69,49.03) and (573.83,60.18) .. (573.83,73.92) -- (573.83,148.59) .. controls (573.83,162.34) and (562.69,173.48) .. (548.94,173.48) -- (107.79,173.48) .. controls (94.04,173.48) and (82.9,162.34) .. (82.9,148.59) -- cycle ;
		\draw [color={rgb, 255:red, 0; green, 0; blue, 0 }  ,draw opacity=1 ][fill={rgb, 255:red, 208; green, 2; blue, 27 }  ,fill opacity=1 ]   (116.5,96.49) -- (107.5,107) ;

		\draw [color={rgb, 255:red, 0; green, 0; blue, 0 }  ,draw opacity=1 ][fill={rgb, 255:red, 208; green, 2; blue, 27 }  ,fill opacity=1 ]   (117.5,125.49) -- (107.5,117) ;

		\draw    (246,49.03) -- (246.7,174.38) ;

		\draw    (413.4,49.03) -- (412.5,173.59) ;

		\draw [color={rgb, 255:red, 0; green, 0; blue, 0 }  ,draw opacity=1 ]   (187.83,106) -- (196.62,96.49) ;

		\draw [color={rgb, 255:red, 0; green, 0; blue, 0 }  ,draw opacity=1 ]   (216.81,111.66) -- (193.83,112) ;

		\draw   (89.17,112.56) .. controls (89.17,93.47) and (104.64,78) .. (123.73,78) .. controls (142.82,78) and (158.29,93.47) .. (158.29,112.56) .. controls (158.29,131.65) and (142.82,147.13) .. (123.73,147.13) .. controls (104.64,147.13) and (89.17,131.65) .. (89.17,112.56) -- cycle ;
		\draw   (170.17,111.56) .. controls (170.17,92.47) and (185.64,77) .. (204.73,77) .. controls (223.82,77) and (239.29,92.47) .. (239.29,111.56) .. controls (239.29,130.65) and (223.82,146.13) .. (204.73,146.13) .. controls (185.64,146.13) and (170.17,130.65) .. (170.17,111.56) -- cycle ;
		\draw   (253.17,111.56) .. controls (253.17,92.47) and (268.64,77) .. (287.73,77) .. controls (306.82,77) and (322.29,92.47) .. (322.29,111.56) .. controls (322.29,130.65) and (306.82,146.13) .. (287.73,146.13) .. controls (268.64,146.13) and (253.17,130.65) .. (253.17,111.56) -- cycle ;
		\draw   (334.17,111.56) .. controls (334.17,92.47) and (349.64,77) .. (368.73,77) .. controls (387.82,77) and (403.29,92.47) .. (403.29,111.56) .. controls (403.29,130.65) and (387.82,146.13) .. (368.73,146.13) .. controls (349.64,146.13) and (334.17,130.65) .. (334.17,111.56) -- cycle ;
		\draw   (419.17,111.56) .. controls (419.17,92.47) and (434.64,77) .. (453.73,77) .. controls (472.82,77) and (488.29,92.47) .. (488.29,111.56) .. controls (488.29,130.65) and (472.82,146.13) .. (453.73,146.13) .. controls (434.64,146.13) and (419.17,130.65) .. (419.17,111.56) -- cycle ;
		\draw   (500.17,111.56) .. controls (500.17,92.47) and (515.64,77) .. (534.73,77) .. controls (553.82,77) and (569.29,92.47) .. (569.29,111.56) .. controls (569.29,130.65) and (553.82,146.13) .. (534.73,146.13) .. controls (515.64,146.13) and (500.17,130.65) .. (500.17,111.56) -- cycle ;
		\draw [color={rgb, 255:red, 0; green, 0; blue, 0 }  ,draw opacity=1 ][fill={rgb, 255:red, 208; green, 2; blue, 27 }  ,fill opacity=1 ]   (139.5,115.49) -- (130.5,126) ;

		\draw [color={rgb, 255:red, 0; green, 0; blue, 0 }  ,draw opacity=1 ][fill={rgb, 255:red, 208; green, 2; blue, 27 }  ,fill opacity=1 ]   (280.5,96.49) -- (271.5,107) ;

		\draw [color={rgb, 255:red, 0; green, 0; blue, 0 }  ,draw opacity=1 ][fill={rgb, 255:red, 208; green, 2; blue, 27 }  ,fill opacity=1 ]   (295.83,98) -- (305.46,107.45) ;

		\draw [color={rgb, 255:red, 0; green, 0; blue, 0 }  ,draw opacity=1 ][fill={rgb, 255:red, 208; green, 2; blue, 27 }  ,fill opacity=1 ]   (287.5,98.49) -- (287.81,121.66) ;

		\draw [color={rgb, 255:red, 0; green, 0; blue, 0 }  ,draw opacity=1 ]   (518.83,106) -- (527.62,96.49) ;

		\draw [color={rgb, 255:red, 0; green, 0; blue, 0 }  ,draw opacity=1 ]   (550.62,106.49) -- (541.83,97) ;

		\draw [color={rgb, 255:red, 0; green, 0; blue, 0 }  ,draw opacity=1 ][fill={rgb, 255:red, 208; green, 2; blue, 27 }  ,fill opacity=1 ]   (446.5,96.49) -- (437.5,107) ;

		\draw [color={rgb, 255:red, 0; green, 0; blue, 0 }  ,draw opacity=1 ][fill={rgb, 255:red, 208; green, 2; blue, 27 }  ,fill opacity=1 ]   (447.5,125.49) -- (437.5,117) ;

		\draw [color={rgb, 255:red, 0; green, 0; blue, 0 }  ,draw opacity=1 ][fill={rgb, 255:red, 208; green, 2; blue, 27 }  ,fill opacity=1 ]   (438.83,112) -- (467.81,111.66) ;

		\draw [color={rgb, 255:red, 0; green, 0; blue, 0 }  ,draw opacity=1 ][fill={rgb, 255:red, 208; green, 2; blue, 27 }  ,fill opacity=1 ]   (469.5,115.49) -- (460.5,126) ;

		\draw [color={rgb, 255:red, 0; green, 0; blue, 0 }  ,draw opacity=1 ]   (351.83,106) -- (360.62,96.49) ;

		\draw [color={rgb, 255:red, 0; green, 0; blue, 0 }  ,draw opacity=1 ]   (380.81,111.66) -- (357.83,112) ;

		\draw [color={rgb, 255:red, 0; green, 0; blue, 0 }  ,draw opacity=1 ]   (546.81,111.66) -- (523.83,112) ;

		\draw (205,89) node  [font=\small,color={rgb, 255:red, 0; green, 0; blue, 0 }  ,opacity=1 ]  {$x_{5}$};
		\draw (185,109) node  [font=\small,color={rgb, 255:red, 0; green, 0; blue, 0 }  ,opacity=1 ]  {$x_{6}$};
		\draw (43,102) node    {${}$};
		\draw (100.79,109.83) node  [font=\small,color={rgb, 255:red, 0; green, 0; blue, 0 }  ,opacity=1 ]  {$x_{2}$};
		\draw (125.33,129.12) node  [font=\small,color={rgb, 255:red, 0; green, 0; blue, 0 }  ,opacity=1 ]  {$x_{3}$};
		\draw (124.96,88.9) node  [font=\small,color={rgb, 255:red, 0; green, 0; blue, 0 }  ,opacity=1 ]  {$x_{1}$};
		\draw (148.5,109.58) node  [font=\small,color={rgb, 255:red, 0; green, 0; blue, 0 }  ,opacity=1 ]  {$x_{4}$};
		\draw (264.79,109.83) node  [font=\small,color={rgb, 255:red, 0; green, 0; blue, 0 }  ,opacity=1 ]  {$x_{2}$};
		\draw (289.33,129.12) node  [font=\small,color={rgb, 255:red, 0; green, 0; blue, 0 }  ,opacity=1 ]  {$x_{3}$};
		\draw (288.96,88.9) node  [font=\small,color={rgb, 255:red, 0; green, 0; blue, 0 }  ,opacity=1 ]  {$x_{1}$};
		\draw (536,89) node  [font=\small,color={rgb, 255:red, 0; green, 0; blue, 0 }  ,opacity=1 ]  {$x_{5}$};
		\draw (516,109) node  [font=\small,color={rgb, 255:red, 0; green, 0; blue, 0 }  ,opacity=1 ]  {$x_{6}$};
		\draw (479,110) node  [font=\small,color={rgb, 255:red, 0; green, 0; blue, 0 }  ,opacity=1 ]  {$x_{4}$};
		\draw (430.79,109.83) node  [font=\small,color={rgb, 255:red, 0; green, 0; blue, 0 }  ,opacity=1 ]  {$x_{2}$};
		\draw (455.33,129.12) node  [font=\small,color={rgb, 255:red, 0; green, 0; blue, 0 }  ,opacity=1 ]  {$x_{3}$};
		\draw (454.96,88.9) node  [font=\small,color={rgb, 255:red, 0; green, 0; blue, 0 }  ,opacity=1 ]  {$x_{1}$};
		\draw (226,110) node  [font=\small,color={rgb, 255:red, 0; green, 0; blue, 0 }  ,opacity=1 ]  {$x_{7}$};
		\draw (369,89) node  [font=\small,color={rgb, 255:red, 0; green, 0; blue, 0 }  ,opacity=1 ]  {$x_{5}$};
		\draw (349,109) node  [font=\small,color={rgb, 255:red, 0; green, 0; blue, 0 }  ,opacity=1 ]  {$x_{6}$};
		\draw (390,110) node  [font=\small,color={rgb, 255:red, 0; green, 0; blue, 0 }  ,opacity=1 ]  {$x_{7}$};
		\draw (313.5,109.58) node  [font=\small,color={rgb, 255:red, 0; green, 0; blue, 0 }  ,opacity=1 ]  {$x_{4}$};
		\draw (560,109) node  [font=\small,color={rgb, 255:red, 0; green, 0; blue, 0 }  ,opacity=1 ]  {$x_{7}$};
		\draw (164,160.26) node   [align=left] {(a)};
		\draw (328,160.26) node   [align=left] {(b)};
		\draw (494,160.26) node   [align=left] {(c)};

		\end{tikzpicture}
		\caption{Decompositions of the partition $\pi$: (a) 	$\mathfrak{a}_{\pi}$, (b) 	$\mathfrak{a}'_{\pi}$, and (c) $	\mathfrak{a}''_{\pi}$.}\label{figura}
	\end{figure}
	
	Decompositions $\mathfrak{a}_{\pi}$ and $\mathfrak{a}'_{\pi}$ are minimal, while decomposition $\mathfrak{a}''_{\pi}$ is not. To show this, let us lean on the following graphical illustration of $\pi$, where for each atom in the decomposition, we placed a segment of line connecting the only two elements in its only non-singleton block.
	
	
	\tikzset{every picture/.style={line width=0.75pt}} 
	
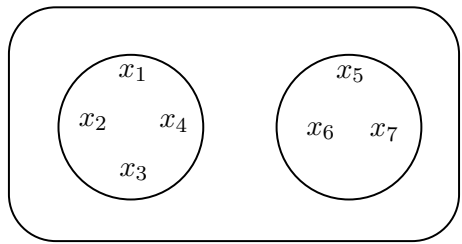
\vspace{0.5cm}
	\hspace{2.5cm}
	\begin{tikzpicture}[x=0.75pt,y=0.75pt,yscale=-1,xscale=1]
	
	\draw   (158.27,53.62) .. controls (158.27,40.59) and (168.84,30.02) .. (181.87,30.02) -- (361.69,30.02) .. controls (374.73,30.02) and (385.29,40.59) .. (385.29,53.62) -- (385.29,124.42) .. controls (385.29,137.46) and (374.73,148.02) .. (361.69,148.02) -- (181.87,148.02) .. controls (168.84,148.02) and (158.27,137.46) .. (158.27,124.42) -- cycle ;
	\draw   (183.27,90.51) .. controls (183.27,70.35) and (199.62,54) .. (219.78,54) .. controls (239.95,54) and (256.29,70.35) .. (256.29,90.51) .. controls (256.29,110.68) and (239.95,127.02) .. (219.78,127.02) .. controls (199.62,127.02) and (183.27,110.68) .. (183.27,90.51) -- cycle ;
	\draw   (293.27,90.51) .. controls (293.27,70.35) and (309.62,54) .. (329.78,54) .. controls (349.95,54) and (366.29,70.35) .. (366.29,90.51) .. controls (366.29,110.68) and (349.95,127.02) .. (329.78,127.02) .. controls (309.62,127.02) and (293.27,110.68) .. (293.27,90.51) -- cycle ;
	
	\draw (200.79,87.83) node  [font=\small,color={rgb, 255:red, 0; green, 0; blue, 0 }  ,opacity=1 ]  {$x_{2}$};
	\draw (221.33,113.12) node  [font=\small,color={rgb, 255:red, 0; green, 0; blue, 0 }  ,opacity=1 ]  {$x_{3}$};
	\draw (220.96,62.9) node  [font=\small,color={rgb, 255:red, 0; green, 0; blue, 0 }  ,opacity=1 ]  {$x_{1}$};
	\draw (241.5,88.58) node  [font=\small,color={rgb, 255:red, 0; green, 0; blue, 0 }  ,opacity=1 ]  {$x_{4}$};
	\draw (330.6,63.66) node  [font=\small,color={rgb, 255:red, 0; green, 0; blue, 0 }  ,opacity=1 ]  {$x_{5}$};
	\draw (347.68,93.1) node  [font=\small,color={rgb, 255:red, 0; green, 0; blue, 0 }  ,opacity=1 ]  {$x_{7}$};
	\draw (315.61,92.37) node  [font=\small,color={rgb, 255:red, 0; green, 0; blue, 0 }  ,opacity=1 ]  {$x_{6}$};

	\end{tikzpicture}
	
	Notice that, according to the definition of join, the atoms in any atomic decomposition of $\pi$ should ensure the existence of a sequence of line segments for any two elements in the same block of $\pi$. If the decomposition is minimal, then only one of such sequences exists for any two given elements lying in the same block of $\pi$.
	It can be straightforwardly seen now that all atoms in $\mathfrak{a}_{\pi}$ and $\mathfrak{a}'_{\pi}$ are needed, while in $\mathfrak{a}''_{\pi}$ we can remove one atom from among the two atoms $\pi_{x_2x_3}$,  $\pi_{x_2x_4}$, and $\pi_{x_3x_4}$, and one atom from among of the atoms $\pi_{x_5x_6}$, $\pi_{x_5x_7}$, and $\pi_{x_6x_7}$.
	
\subsection*{Contributions of the article}
	Firstly, this article investigates some basic properties of the atomic decomposition in the lattice of partitions, including those introduced by D.D. Anderson, D. F. Anderson and M. Zafrullah for the systematic study of commutative algebraic structures in factorization theory. Then, we focus on
	the key characteristic features of the function $\mathfrak{N}:\Pi(X)\rightarrow\mathbb{N}$, which maps each partition $\pi$ of $X$ into the number of minimal atomic decompositions of $\pi$. Secondly, assuming that a subset $\mathcal{R}$ of the set atoms $\mathcal{A}$ is given, referred to as the set of red atoms, we derive a recursive formula for $\bm{\pi}(X,j,s,\mathcal{R})$: the total number of partitions of $X$ lying at the $j$th level of $\Pi(X)$ that can be expressed as the join of exactly $s$ red atoms. (The term ``red atoms'' is used to emphasize that only a designated subset of atoms is considered, rather than the full set of atoms in the lattice.)
	
	From a geometrical perspective, atomic decompositions can be understood as the process by which higher-dimensional structures, such as lines and planes, are formed by combining the smallest ``building blocks" of the geometry, typically referred to as points. Emphasizing atomic decomposition in geometric lattices highlights the essence of constructing complex structures from simple, indivisible elements, reflecting both algebraic and geometric viewpoints.
	
To facilitate accessibility and understanding, Table \ref{tab:notation} compiles the main notations used throughout this article.
	
\begin{table}[h!]
	\centering
	\renewcommand{\arraystretch}{1.3}
	\begin{tabularx}{\textwidth}{|c|>{\raggedright\arraybackslash}p{4cm}|>{\raggedright\arraybackslash}X|}
		\hline
		\textbf{No.} & \textbf{Notation} & \textbf{Description} \\
		\hline
			\rownumber & \( \Pi(X) \) & Lattice of partitions of the finite set $X$. \\
		\rownumber & \( \mathcal{A} \) & Set of atoms of \(\Pi(X)\). \\
		\rownumber & \( \mathcal{A}(\pi) \) & Set of atoms that refine the partition \( \pi \). \\
			\rownumber & \( M(\pi) \) & Set of minimal atomic decompositions of the partition $\pi$. \\
		\rownumber & \( \mathfrak{N}(\pi) \) & Number of minimal atomic decompositions of \(\pi\). \\
			\rownumber & \( \mathcal{R} \) & Specified subset of atoms referred to as ``red atoms''. \\
		\rownumber & \( \Pi(X,\mathcal{J}) \) & Partitions expressible as \( \bigvee J \) for \( J \in \mathcal{J} \subseteq 2^{\mathcal{R}} \). \\
		\rownumber & \( \Pi(X,\mathcal{R}) \) & Set of partitions formed as joins of subsets of \(\mathcal{R}\). \\
		\rownumber & \( \Pi(X,*,s,\mathcal{R}) \) & Partitions expressible as joins of exactly \( s \) atoms from \( \mathcal{R} \). \\
		\rownumber & \( \Pi(X,j,s,\mathcal{R}) \) & Subset of \( \Pi(X,*,s,\mathcal{R}) \) consisting of those partitions of rank \( j \). \\
		\rownumber & \( \Pi(X,j,s,\mathcal{R})_{xx'} \) & Partitions \( \pi \in \Pi(X,j,s,\mathcal{R}) \) refined by the atom \( \pi_{xx'} \). \\
		\rownumber & \( \Pi(X,j,s,\mathcal{R})_{x|x'} \) & Partitions \( \pi \in \Pi(X,j,s,\mathcal{R}) \) that are not refined by the atom \( \pi_{xx'} \). \\
		\rownumber & \( \boldsymbol{\pi}(X,j,s,\mathcal{R}) \) & Cardinality of \( \Pi(X,j,s,\mathcal{R}) \). \\
		\hline
	\end{tabularx}
	\caption{Notation used for atomic decompositions in the partition lattice}
	\label{tab:notation}
\end{table}

	\section{Basic facts about atomic decompositions in $\Pi(X)$}
	
	Our first theorem enables us to leverage well-known results from graph theory to describe fundamental properties of the atomic decomposition of partitions.
	
	Along this paper we will use the maps:
	\begin{enumerate}
		\item the map that takes a partition $\pi\in\Pi(X)$ and returns the labeled non-directed graph $G_{\pi}=(X,E_{\pi})$ on $X$, where $\{x,x'\}\in E_{\pi}$ if and only if $\pi_{xx'}\in \mathcal{A}(\pi)$.
		\item the map that takes a labeled undirected graph $G=(X,E)$ on $X$ and returns the partition $\pi_G=\bigvee\{\pi_{xx'}\,:\, \{x,x'\}\in E\}$.
	\end{enumerate}
	
	\begin{Theorem}\label{Bijection} Let $\pi\in\Pi(X)$ be an arbitrary partition of $X$. Then:
		\begin{enumerate}
			\item There is a bijection between the set of all atomic decompositions of $\pi$ and the set of all spanning subgraphs $G$ of $G_{\pi}$ such that $\pi_G=\pi$.
			\item There is a bijection between the set of all minimal atomic decompositions of $\pi$ and the set of all spanning forests of $G_{\pi}$.
		\end{enumerate}		 	
	\end{Theorem}
	\begin{proof}	
		Let $\mathfrak{a}_{\pi}$ be an atomic decomposition of the partition $\pi = \{\textsc{b}_1, \textsc{b}_2, \ldots, \textsc{b}_k\}$. We construct a labeled spanning subgraph $\mathfrak{f}_{\mathfrak{a}_{\pi}}$ of $G_{\pi}$ as follows:
		The vertex set of $\mathfrak{f}_{\mathfrak{a}_{\pi}}$ is $X$. Two vertices $x, x' \in X$ are connected by an edge in $\mathfrak{f}_{\mathfrak{a}_{\pi}}$ if and only if the atom $\pi_{xx'} \in \mathfrak{a}_{\pi}$. By construction, since each atom in $\mathfrak{a}_{\pi}$ is determined by a unique pair of elements forming its only nonsingleton block, $\mathfrak{f}{\mathfrak{a}_{\pi}}$ is necessarily loop-free. Moreover, by virtue of definition of the join operator, the graph $\mathfrak{f}_{\mathfrak{a}_{\pi}}$ respects the block structure of $\pi$, meaning that all edges connect elements within the same block of $\pi$, and any two vertices in the same block of $\pi$ are connected by a path in $\mathfrak{f}_{\mathfrak{a}_{\pi}}$. Thus, a subgraph $C$ is a connected component of $\mathfrak{f}_{\mathfrak{a}_{\pi}}$ if and only if  there is a block $\textsc{b}_s$, $1\leq s\leq k$, of $\pi$ such that $C$ is a connected graph on the elements of $\textsc{b}_s$. Therefore, $\mathfrak{f}_{\mathfrak{a}_{\pi}}$ is a spanning subgraph of $G_{\pi}$ such that $\pi_{\mathfrak{f}_{\mathfrak{a}_{\pi}}}=\pi$.
		
		Let $\phi_{\pi}$ be the function that assigns the spanning subgraph $\mathfrak{f}_{\mathfrak{a}_{\pi}}$ to the atomic decomposition $\mathfrak{a}_{\pi}$ of $\pi$. We argue that $\phi_{\pi}$ is a bijection. Indeed, suppose $\mathfrak{a}_{\pi}$ and $\mathfrak{\widehat{a}}_{\pi}$ yield the same spanning subgraph of $G_{\pi}$, i.e., $\mathfrak{f}_{\mathfrak{a}_{\pi}}=\mathfrak{f}_{\mathfrak{\widehat{a}}_{\pi}}$. Then, $\mathfrak{f}_{\mathfrak{a}_{\pi}}$ and $\mathfrak{f}_{\mathfrak{\widehat{a}}_{\pi}}$ have identical edges, implying that $\mathfrak{a}_{\pi}$ and $\mathfrak{\widehat{a}}_{\pi}$ consist of the same atoms. Thus, $\mathfrak{a}_{\pi} = \mathfrak{\widehat{a}}_{\pi}$, proving injectivity. Let us consider now a spanning subgraph $\mathfrak{f}_{\pi}$ of $G_{\pi}$. Define $\mathfrak{a}$ as the set of atoms $\pi_{xx'}$ corresponding to the edges $\{x,x'\}$ of $\mathfrak{f}_{\pi}$. Then, $\mathfrak{f}_{\pi}$ is precisely $\phi_{\pi}(\mathfrak{a})$, proving surjectivity.
		
		To complete the proof, note that a decomposition $\mathfrak{a}_{\pi}$ is minimal if and only if $\mathfrak{f}_{\mathfrak{a}_{\pi}}$ is a forest. Any cycle in $\mathfrak{f}_{\mathfrak{a}{\pi}}$ implies the existence of redundant atoms in $\mathfrak{a}_{\pi}$, contradicting minimality. Conversely, removing superfluous atoms in a non-minimal decomposition corresponds to removing edges from $\mathfrak{f}_{\mathfrak{a}_{\pi}}$, resulting in a forest. 
	\end{proof}

	\begin{Corollary} \label{closedunderunion}
		The union of atomic decompositions of a partition $\pi$ is also an atomic decomposition of this partition.
	\end{Corollary}
	\begin{proof} This follows immediately from Theorem \ref{Bijection} and the fact that the union of spanning subgraphs of a graph is also a spanning subgraph of this graph.
	\end{proof}
	
	\begin{Corollary} \label{Lemmauniqueseqeunce} The subset
		$\mathfrak{a}_{\pi}\subseteq \mathcal{A}(\pi)$ is an atomic decomposition of the partition $\pi$ if and only if for every  atom 
		$\pi_{xx'}$,  $\pi_{xx'}\preceq \pi$ if and only if there exists a  sequence   $\pi_{x_0x_1}, \pi_{x_1x_2}, \ldots,\pi_{x_{s-1}x_s}$  of distinct atoms in $\mathfrak{a}_{\pi}$ such that $x_0=x$  and  $x_s=x'$.
		Furthermore, $\mathfrak{a}_{\pi}$ is minimal if and only if, for every atom $\pi_{xx'}\preceq \pi$, the sequence above is unique.
	\end{Corollary}
	\begin{proof} From Theorem \ref{Bijection}, we know that $\mathfrak{a}_{\pi}$ is an atomic decomposition of $\pi$ if and only if $\mathfrak{f}_{\mathfrak{a}_{\pi}}$ is a spanning subgraph of $G_{\pi}$. Since an atom $\pi_{xx'}$ refines $\pi$ (i.e., $\pi_{xx'} \preceq \pi$) if and only if $x$ and $x'$ lie in the same block of $\pi$, it follows that $\pi_{xx'} \preceq \pi$ if and only if there exists a path in $\mathfrak{f}_{\mathfrak{a}_{\pi}}$ connecting $x$ and $x'$. The edges of this path correspond to the desired sequence of atoms $\pi_{x_0x_1}, \pi_{x_1x_2}, \ldots, \pi_{x_{s-1}x_s}$.
		
		Furthermore, $\mathfrak{a}_{\pi}$ is minimal if and only if $\mathfrak{f}_{\mathfrak{a}_{\pi}}$ is a forest. In this case, any path between two vertices $x$ and $x'$ in $\mathfrak{f}_{\mathfrak{a}_{\pi}}$ is unique, ensuring that the sequence of atoms connecting $x$ and $x'$ is also unique. Conversely, if the sequence of atoms is unique for every pair $x$ and $x'$ such that $\pi_{xx'} \preceq \pi$, then $\mathfrak{f}_{\mathfrak{a}_{\pi}}$ must be a forest, and $\mathfrak{a}_{\pi}$ is minimal.
	\end{proof}	
	
	\begin{Corollary}\label{allminimalsamesize}
		All minimal atomic decompositions of a partition $\pi$ consist of the same number of atoms, which is $\displaystyle \sum_{i=1}^{|\pi|} \left(n_i-1\right)=n-|\pi|$, where $n_i$ stands for the number of elements in the $i$th block of $\pi$. Additionally, the number $m$ of atoms in any atomic decomposition of the partition $\pi$ satisfies 
		\begin{equation}
	n-|\pi| \leq m\leq \sum_{i=1}^{|\pi|} \binom{n_i}{2}.
		\end{equation}
	\end{Corollary}
	\begin{proof}  The number of atoms in a minimal atomic decomposition of $\pi$ is determined by Theorem \ref{Bijection}, which ensures that $\mathfrak{f}_{\mathfrak{a}{\pi}}$ is a spanning forest of $G_{\pi}$. Since every connected component of $\mathfrak{f}_{\mathfrak{a}{\pi}}$ corresponds to a block of $\pi$ and is a tree, the number of edges (and hence atoms) in $\mathfrak{f}_{\mathfrak{a}{\pi}}$ is precisely $\sum_{i=1}^{|\pi|} (n_i - 1)=n-|\pi|$.
		
		For an arbitrary atomic decomposition, the lower bound follows because it must contain at least the atoms of a minimal decomposition to preserve connectivity within each block. The upper bound is attained for the atomic decomposition $\mathcal{A}(\pi)$, leading to the sum $\sum_{i=1}^{|\pi|} \binom{n_i}{2}$.
	\end{proof}
	
	\begin{Corollary}\label{not minimal}
		The union of different minimal atomic decompositions of a partition $\pi$ is not minimal.
	\end{Corollary}
	\begin{proof} From Corollary \ref{allminimalsamesize}, all minimal atomic decompositions of $\pi$ have the same number of atoms. This implies that the union of two distinct minimal atomic decompositions strictly contains both decompositions. Consequently, the union exceeds the size of a minimal decomposition, violating minimality.
	\end{proof}
	
	Our next result provides an analytical expression for the function $\mathfrak{N}$.
	
	\begin{Corollary} Let $\mathfrak{N}:\Pi(X)\rightarrow\mathbb{N}$ be the function which maps each partition $\pi$ of $X$ to the number of its minimal atomic decompositions. Then, for every partition $\pi = \{\textsc{b}_1, \textsc{b}_2, \ldots, \textsc{b}_k\}$, with $n_i=|\textsc{b}_i|$ and $k_0$ denoting the number of nonsingleton blocks of $\pi$, the following holds:
		\begin{equation}\label{formula}
		\mathfrak{N}(\pi)=\prod_{i=1}^{k_0} n_i^{n_i-2}.
		\end{equation} 
	\end{Corollary}
	\begin{proof}
		The graph $G_{\pi}$ has $k$ connected components, each being a complete graph on the elements of a block of $\pi$. It is well known that the number of spanning forests of a graph is given by the product of the number of spanning trees of its connected components. Specifically, for $G_{\pi}$, this is $\prod_{i=1}^k t_i$, 
		where $t_i$ is the number of spanning trees of the $i$th connected component.
		
		Using Cayley's formula, which states that a complete graph on $n_i\geq2$ vertices has $n_i^{n_i - 2}$ spanning trees, we deduce that if the $i$th connected component of $G_{\pi}$ contains $n_i$ vertices, the number of spanning trees of that component is $n_i^{n_i - 2}$. Finally, applying Theorem \ref{Bijection}, which establishes a correspondence between minimal atomic decompositions of $\pi$ and spanning forests of $G_{\pi}$, we conclude that	
		
		\begin{equation}
		\mathfrak{N}(\pi)=\prod_{i=1}^{k_0} n_i^{n_i-2}.
		\end{equation}
	\end{proof}

We now turn to analyzing the atomic decompositions of $\Pi(X)$ from an algebraic perspective. To this end, and for the sake of completeness, we shall briefly review some relevant algebraic concepts.

A \emph{monoid} is a set equipped with an associative binary operation with an identity element.  Thus, $\Pi(X)$ with the join operator $\vee$ is a commutative monoid. For the systematic study of the  atomic decompositions in monoid, usually referred to as factorizations, D.D. Anderson, D.D. Anderson and M. Zafrullah introduced the following properties:

\begin{enumerate}
	\item  A subset $I$ of a monoid $M$ is an \emph{ideal} of $M$ provided that $I + M \subseteq I$. The ideal $I$ is \emph{principal} if $I = x + M$ for some $x \in M$. The monoid $M$ satisfies the \emph{ascending chain condition on principal ideals} (or \emph{ACCP}) if each increasing sequence of principal ideals of $M$ eventually stabilizes.
	\item An atomic monoid $M$ is said to be
	a finite factorization monoid (FFM) if every element admits only finitely many factorizations. 
	\item An atomic monoid $M$ is said to be a bounded factorization monoid (BFM) if for every element there is an upper bound for the number
	of atoms (counting repetitions) in each of its factorizations.
	\item An atomic
	monoid $M$ is said to be half-factorial (HFM) if any two factorizations of the same element have the same number of
	atoms (counting repetitions).
	\item An atomic
	monoid $M$ is said to be unique factorization monoid (UFM) if every element has a unique factorization.
\end{enumerate}
	
It is well-known that HFM implies BFM, which in turn implies ACCP, and consequently implies atomic. Additionally, UFM implies FFM, which in turn implies BFM. 

Since $\Pi(X)$ is finite, it satisfies the ACCP condition. However, the join operator is idempotent, i.e., for every partition $\pi\in\Pi(X)$, we have $\pi\vee\pi=\pi$, which means that we can add an atom to an atomic decomposition as many times as we wish. This behavior causes $\Pi(X)$ to fail all the other conditions.

In this context, however, it is more beneficial to disregard repeated atoms, as they are never necessary and may hinder our ability to identify distinguishing features among finite lattices. For this reason, we will consider variants of the conditions above in which repeated atoms are not allowed.

\begin{Proposition} The lattice $\Pi(X)$ satisfies the ACCP, FFM and BFM conditions, and fails to satisfy HFM and UFM conditions. (All conditions here do not allow repeated atoms.)
\end{Proposition}

\begin{proof}
  Given a partition $\pi\in\Pi(X)$, let $m_{\pi}$ denote the number of connected components of the graph $G_{\pi}$. Theorem \ref{Bijection} states that the number of atomic decompositions of $\pi$ is equal to the number of spanning subgraphs of $G_{\pi}$ with $m_{\pi}$ connected components. Therefore, $\Pi(X)$ satisfies the FFM condition, and consequently, the BFM and ACCP conditions. 
  
  However, the number of spanning subgraphs of the $G_{\pi}$ depends on the sizes of the blocks of $\pi$, which form a partition of the integer $n$ into $m_{\pi}$ parts. Consequently, there exist partitions whose corresponding partitions of the integer $n$ differ and therefore yield different numbers of atomic decompositions --- for instance, an atom and any partitions that covers it.
\end{proof}

We introduce here a novel condition of the same kind, which enables us to distinguish among ranked lattices, although it does not extend to general monoids.

A ranked lattice $L$ is said to satisfy the half-factorial ranked lattice (HFRL) condition if, for any two elements $x,x'\in L$, the equality $\text{rank}(x)=\text{rank}(x')$ implies that $x$ and $x'$ have the same number of atomic decompositions. This property, for instance, distinguishes the lattice of partitions from the lattice of subspaces of a vector space over a finite field, since the former fails to satisfy it, while the latter does.  

\begin{Lemma} $\Pi(X)$ fails to satisfy HFRL.
	\end{Lemma}
	\begin{proof}
		Any two partitions $\pi$ and $\pi'$ at the same level of $\Pi(X)$ have the same number of blocks, which means that their respective sequences of block sizes correspond to partitions of the integer $n$ into the same number of parts. However, we can always choose $\pi$ and $\pi'$ with corresponding partitions of $n$ that differ in a way that leads to a different number of atomic decompositions. For example, suppose $\pi$ and $\pi'$ each have three blocks, with block size sequences 1+2+2 and 1+1+3. The number of atomic decompositions is 2 and 4, respectively.
	\end{proof}

	Aiming at investigating the other properties of minimal atomic decompositions by analyzing the behavior of the function $\mathfrak{N}$ along chains in $\Pi(X)$, we conclude this section with two outcomes that shed some light on the topology of the lattice of partitions.
	
	A \emph{chain} in the lattice of partitions $\Pi(X)$ is a sequence of partitions in which every partition refines all partitions  that succeed it in the sequence; i.e., $\pi_0\preceq\pi_1\preceq\ldots\preceq\pi_s$.  Notice that, in particular, for every pair of partitions such that $\pi\prec\pi'$,  there is a chain in which every partition is covered by its immediate successor: $\pi=\pi_0\sqsubset\pi_1\sqsubset\pi_2\sqsubset\ldots\sqsubset\pi_s=\pi'$. Let us suppose that $\pi\sqsubset \pi'$ and let $\textsc{b}_i$ and $\textsc{b}_j$ denote the two blocks of $\pi$ needed to be merged to obtain $\pi'$. Note that $\pi' = \pi\vee\pi_{x_ix_j}$ if and only if 
	$x_i\in \textsc{b}_i$ and $x_j\in \textsc{b}_j$. Let  $x_i$ and $x_j$ be elements satisfying this condition and let us define the mapping $\Psi_{x_ix_j}$ that to every minimal atomic decomposition $\mathfrak{a}_{\pi}$ of $\pi$ assigns $\mathfrak{a}_{\pi'}:=\mathfrak{a}_{\pi}\cup \{\pi_{x_ix_j}\}$. 
	
	\begin{Proposition}\label{addinganatominjections}
		The map $\Psi_{x_ix_j}$ is an injective function whose range lies within the set of all possible atomic decompositions of $\pi'$. Moreover, if we consider a different pair of elements $y_i\in \textsc{b}_i$ and $y_j\in \textsc{b}_j$, the respective ranges of $\Psi_{x_ix_j}$ and $\Psi_{y_iy_j}$ do not overlap.
	\end{Proposition}
	\begin{proof} To verify that $\Psi_{x_ix_j}$ is an injective function, let us consider two atomic decompositions of $\mathfrak{a}_{\pi}$ and $\hat{\mathfrak{a}}_{\pi}$ of the partition $\pi$ such that $\Psi_{x_ix_j}\left(\mathfrak{a}_{\pi}\right)=\Psi_{x_ix_j}\left(\hat{\mathfrak{a}}_{\pi}\right)$. By definition of $\Psi_{x_ix_j}$, 
		this implies that $\mathfrak{a}_{\pi}\cup \{\pi_{x_ix_j}\} = \hat{\mathfrak{a}_{\pi}}\cup \{\pi_{x_ix_j}\}$. Since both sets contain the same atoms, $\mathfrak{a}_{\pi}=\hat{\mathfrak{a}_{\pi}}$. Thus, $\Psi_{x_ix_j}$ is injective.
		
		To prove that  the ranges of $\Psi_{x_ix_j}$ and  $\Psi_{y_iy_j}$ do not overlap for a different pair of elements $y_i\in \textsc{b}_i$ and $y_j\in \textsc{b}_j$, observe the following: every atomic decomposition returned by $\Psi_{x_ix_j}$ contains the atom $\pi_{x_ix_j}$ but not the atom $\pi_{y_iy_j}$, and conversely, every atomic decomposition returned by $\Psi_{y_iy_j}$ contains the atom $\pi_{y_iy_j}$ but not the atom $\pi_{x_ix_j}$. Since no atomic decomposition can simultaneously contain both $\pi_{x_ix_j}$ and $\pi_{y_iy_j}$, the ranges of $\Psi_{x_ix_j}$ and  $\Psi_{y_iy_j}$ are disjoint.
	\end{proof}
	
	It is worthy to remark that not all minimal atomic decomposition of $\pi'$ can be obtained this way from a minimal atomic decomposition of $\pi$. Let $X=\{x_1,x_2,x_3,x_4,x_5,x_6,x_7\}$ and consider the partitions $\pi=\{\{x_1,x_2,x_3\},\{x_4,x_5,x_6\},\{x_7\}\}$  and $\pi'=\{\{x_1,x_2,x_3,x_4,x_5,x_6\},\{x_7\}\}$. Note that $\pi\sqsubset\pi'$. One can straightforwardly verify that $\mathfrak{a}_{\pi'}=\{\pi_{x_1x_4}, \pi_{x_1x_5},\pi_{x_2x_5}, \pi_{x_2x_6}, \pi_{x_3x_6}\}$ is a minimal atomic decomposition of $\pi'$ which cannot be returned by any of the maps $\Psi_{x_ix_j}$. In particular, note that none of the atoms in this decomposition refines $\pi$.
	
	The next corollary is basically a consequence of the fact that the composition of injective functions is an injective function as well. In spite of its simplicity, we decided to be emphatic about it here because we will need it later.

	\begin{Corollary}\label{compositionofaddinganatom} Let $\pi\prec\pi'$ and $\pi=\pi_0\sqsubset\pi_1\sqsubset\pi_2\sqsubset\ldots\sqsubset\pi_s=\pi'$ any maximal chain from $\pi$ into $\pi'$. If $\pi_{i}=\pi_{i-1} \vee \pi_{a_ib_i}$, then the composition $\Psi_{x_{a_s}x_{b_s}}\circ \Psi_{x_{a_{s-1}}x_{b_{s-1}}}\circ\ldots\circ \Psi_{x_{a_1}x_{b_1}}$ is an injection from the set of minimal atomic decomposition of $\pi$ into the set of minimal atomic decompositions of $\pi'$.
	\end{Corollary}
	
	The previous results offer valuable insights into retrieving minimal atomic decompositions of partitions as we move vertically within the Hasse diagram of $\Pi(X)$. But what about horizontal movement? In other words, is there a relationship between the minimal atomic decompositions of partitions within the same layer of $\Pi(X)$? The following proposition addresses this question and provides some clarity.
	
	\begin{Proposition}\label{permutations}
		Let $\pi=\{\textsc{b}_1, \textsc{b}_2,\ldots,\textsc{b}_p\}$ and $\pi'=\{\textsc{b}'_1, \textsc{b}'_2,\ldots,\textsc{b}'_p\}$ be partitions of $X$ such that there exists a permutation $\tau$ of the index set $\{1,2,\ldots,p\}$ satisfying $n_i=n'_{\tau(i)}$ for every $i\in\{1,2,\ldots,p\}$. Then, for every permutation $\sigma:X\rightarrow X$ of the elements of $X$ such that, for any $x,x'\in X$, $x$ and $x'$ lie in the same block of $\pi$ if and only if $\sigma(x)$ and $\sigma(x')$ lie in the same block of $\pi'$,  the mapping $\Psi_{\sigma}$ which assigns to every minimal atomic decomposition $\mathfrak{a}_{\pi}$ of $\pi$  the minimal atomic decomposition $\mathfrak{a}^{\sigma}_{\pi'}$ of $\pi'$ defined by
		
		$$\pi_{xx'}\in \mathfrak{a}_{\pi} \text{ if, and only if, } \pi_{\sigma(x)\sigma(x')}\in \mathfrak{a}^{\sigma}_{\pi'}$$
		
		is a bijection from the set of all minimal atomic decompositions of $\pi$ onto the set of all minimal atomic decompositions of $\pi'$.
	\end{Proposition}
	\begin{proof}  The proof follows directly by constructing the inverse map of $\Psi_{\sigma}$. Define $\Psi_{\sigma^{-1}}$, which maps each minimal atomic decomposition $\mathfrak{a}_{\pi'}$ of $\pi'$ back to a minimal atomic decomposition $\mathfrak{a}^{\sigma^{-1}}_{\pi}$ of $\pi$, as follows:
		
		$$\pi_{xx'}\in \mathfrak{a}'_{\pi} \text{ if and only if } \pi_{\sigma^{-1}(x)\sigma^{-1}(x')}\in \mathfrak{a}^{\sigma}_{\pi},$$
		
		where $\sigma^{-1}$ is the inverse of $\sigma$. Clearly, $\Psi_{\sigma^{-1}}$ reverses the mapping of $\Psi_{\sigma}$, and thus $\Psi_{\sigma}$ is bijective.  
		
	\end{proof}

	\section{The main characteristic features of the function $\mathfrak{N}$. }\label{PropertiesOfN}
	
	We now derive a series of interesting properties that the function $\mathfrak{N}$ satisfies. For the sake of completeness, we also include some fundamental properties.  
	
	\begin{Theorem}
		\begin{description}
			\item[(\textrm{i})]  $\mathfrak{N}(\pi) = 0$ if and only if $\pi=\textrm{m}_X$. 
			\item[(\textrm{ii})] $\mathfrak{N}(\pi) = 1$ if and only if $\pi\neq \textrm{m}_X$ and has no blocks containing more than two elements of $X$. Especially, if $\pi$ is an atom, $\mathfrak{N}(\pi) = 1$.
			\item[(\textrm{iii})] If $\mathfrak{N}(\pi) = 1$, then $\pi$ has no less than $\lceil n/2\rceil$ blocks ($n$ denotes the number of elements in $X$).
			\item[(\textrm{iv})] Let $\pi = \{\textsc{b}_1, \textsc{b}_2, ..., \textsc{b}_p\}$ and $\pi' = \{\textsc{b}'_1, \textsc{b}'_2, ..., \textsc{b}'_q\}$ be arbitrary partitions of $X$ and let $n_i$ and $n'_j$ denote the number of elements in the $i$th block $\textsc{b}_i$ of $\pi$ and in the $j$th block $\textsc{b}'_j$ of $\pi'$, respectively. If $p = q$ and, there is a permutation $\lambda$ of $\mathbf{p}$ such that $n_i = n'_{\lambda(i)}$, then $\mathfrak{N}(\pi) = \mathfrak{N}(\pi')$. (The converse is not true.) In particular, $\mathfrak{N}$ is not an injection. 
			\item[(\textrm{v})] 
			If $\pi \sqsubset \pi'$ and $\textsc{b}_i$ and $\textsc{b}_j$ are the two blocks of $\pi$ needed to be merged to obtain $\pi'$, then $\mathfrak{N}(\pi') \geq n_i\cdot n_j\cdot \mathfrak{N}(\pi)$. The equality holds if and only if $\textsc{b}_i$ and $\textsc{b}_j$ are singleton blocks.
			\item[(\textrm{vi})] $\mathfrak{N}$ is an isotone function, i.e., $\pi \prec   \pi'$ implies $\mathfrak{N}(\pi) \leq \mathfrak{N}(\pi')$.
			\item[(\textrm{vii})] $\mathfrak{N}(\pi)=n^{n-2}$ if, and only if, $\pi=\textrm{g}_X$.
			\item[(\textrm{viii})] For every partition $\pi$ of $X$, $0\leq\mathfrak{N}(\pi)\leq n^{n-2}$. Thus, $\mathfrak{N}$ is a bounded function.
			\item[(\textrm{ix})] $\mathfrak{N}$ is a supermodular function, i.e., for all partitions $\pi$ and $\pi'$,
			\begin{equation}
			\mathfrak{N}(\pi) +\mathfrak{N}(\pi') \leq \mathfrak{N}(\pi\wedge \pi') +\mathfrak{N}(\pi\vee\pi').
			\end{equation}
			
			\item[(\textrm{x})]The function $d:\Pi(X)\times\Pi(X)\rightarrow \mathbb{R}$
			defined by
			\begin{equation}
			d(\pi,\pi')= \mathfrak{N}(\pi) +\mathfrak{N}(\pi') -2\mathfrak{N}(\pi\wedge \pi')
			\end{equation}
			is a metric on $\Pi(X)$.
		\end{description}
	\end{Theorem}
	
	\begin{proof}
		
		To verify  $(i)$, note that if $\pi\neq \textrm{m}_X$, then at least one block of $\pi$ is a nonsingleton, which implies that $\mathcal{A}(\pi)\neq \emptyset$. Thus, $\mathcal{A}(\pi)$ itself constitutes an atomic decomposition $\mathfrak{a}_{\pi}$ of $\pi$ and, as we have already noted, $\pi$ admits a minimal atomic decomposition. This proves that $\pi\neq \textrm{m}_X$ forces $\mathfrak{N}(\pi)>0$. Conversely, $\mathcal{A}(\textrm{m}_X)= \emptyset$, and therefore $\textrm{m}_X$ admits no atomic decomposition.
		
		Suppose now that a partition $\pi$ admits different minimal atomic decompositions. Lemma \ref{closedunderunion} ensures that the union of two of these decompositions is another atomic decomposition of $\pi$ which, according to Corollary \ref{not minimal}, is not minimal. As a consequence, this decomposition contains a sequence of redundant atoms $\pi_{x_0x_1}, \pi_{x_1x_2}, \ldots, \pi_{x_sx_0}$ with $s\geq 3$. By the definition of join, the elements $x_0, x_1,\ldots, x_s$ belong to the same block of $\pi$. Therefore, if a partition admits more than one minimal atomic decomposition, then this partition has a block with at least three elements. Conversely, let us assume that there is a block $\textsc{b}$ of $\pi$ containing $s\geq 3$ elements. Since the atoms belonging to a given atomic decomposition are collected blockwise, there is no loss of generality in assuming that $\textsc{b}$ is the only block of $\pi$. To fix ideas, let us suppose once again that the elements of $\textsc{b}$ have been enumerated. The decompositions $\mathfrak{a}_{\pi}=\{\pi_{x_0x_i}\,:\, 2\leq i\leq s\}$ and $\hat{\mathfrak{a}}_{\pi}=\{\pi_{x_ix_{i+1}}\,:\, 1\leq i\leq s-1\}$  are different minimal atomic decompositions of $\pi$.  We have thus proven $(ii)$.
		
		Property $(iii)$ follows immediately from the fact that the maximum number of pairings that we can form 
		with the elements of $X$ is $\lceil n/2\rceil$ blocks. So, this is the maximum number of 2-size blocks that a partition of $X$ can have. Since $\mathfrak{N}(\pi)=1$ limits the maximum size of a block to be 2, we have that, for every 2-size block that a partition discards, two singleton blocks are taken, which increases the number of blocks of the partition.
		
		In order to prove $(iv)$, let $\pi=\{\textsc{b}_1, \textsc{b}_2,\ldots,\textsc{b}_p\}$ and $\pi'=\{\textsc{b}'_1, \textsc{b}'_2,\ldots,\textsc{b}'_q\}$ be partitions of $X$ such that $p = q$ and there exists a permutation $\lambda$ of $\mathbf{p}$ such that $n_i = n'_{\lambda(i)}$, then $\mathfrak{N}(\pi) = \mathfrak{N}(\pi')$. 
		Under these conditions, Proposition \ref{permutations} guarantees the existence of a bijective map $\Psi_{\sigma}$, depending on an appropriate permutation $\sigma$ of the elements of $X$, that maps every minimal atomic decomposition of $\pi$ into a minimal atomic decomposition of $\pi'$.
		
		As for $(v)$, let $\pi \sqsubset \pi'$ and let $\textsc{b}_i$ and $\textsc{b}_j$ denote the two blocks of $\pi$ needed to be merged to obtain $\pi'$. Proposition \ref{addinganatominjections} ensures that the mapping $\Psi_{x_ix_j}$ which to every minimal atomic decomposition $\mathfrak{a}_{\pi}$ of $\pi$ assigns the minimal atomic decomposition $\mathfrak{a}_{\pi'}:=\mathfrak{a}_{\pi}\cup \{\pi_{x_ix_j}\}$ of $\pi'$, where $x_i\in \textsc{b}_i$ and $x_j\in \textsc{b}_j$, is an injection. Besides, this proposition asserts that the set of minimal atomic decompositions of $\pi'$ contains the disjunct union of the ranges of the injection $\Psi_{x_ix_j}$, considering all possible combinations of $x_i\in \textsc{b}_i$ and $x_j\in \textsc{b}_j$.  Since each of these ranges contains as many elements as minimal atomic decompositions $\pi$ admits, and there are $n_i\cdot n_j$ possible unordered pairs $x_i\in \textsc{b}_i$ and $x_j\in \textsc{b}_j$, we can conclude that $\mathfrak{N}(\pi') \geq n_i\cdot n_j\cdot \mathfrak{N}(\pi)$.
		
		If $\pi\prec \pi'$, then there is a sequence of partitions $\pi=\pi_0\sqsubset \pi_1 \sqsubset \ldots \sqsubset \pi_m=\pi$. Property $(vi)$ is now an immediate consequence of $(v)$.
		
		To see $(vii)$, we shall start by noticing that if $\pi\neq \textrm{g}_X$, then $\mathfrak{N}(\pi)< \mathfrak{N}(\textrm{g}_X)$. To this end, let us consider $\pi\sqsubset \textrm{g}_X$, and let $\textsc{b}_i$ and $\textsc{b}_j$ be those blocks of $\pi$ that should be merged in order to produce $\textrm{g}_X$. If $\pi\neq\textrm{m}_X$ (otherwise the result follows immediately), then $\textsc{b}_i$ and $\textsc{b}_j$ are not both singleton and, by virtue of $(v)$, $\mathfrak{N}(\textrm{g}_X)\geq n_i\cdot n_j\cdot\mathfrak{N}(\pi)>\mathfrak{N}(\pi)$. Let now $\pi'$ be an arbitrary partition of $X$. There is a partition $\pi\sqsubset\textrm{g}_X$ such that $\pi'\preceq\pi$. Using now Property $(vi)$, we can conclude that $\mathfrak{N}(\pi')\leq\mathfrak{N}(\pi)<\mathfrak{N}(\textrm{g}_X)$. It only remains to prove that $\mathfrak{N}(\textrm{g}_X)=n^{n-2}$, but this derives forthwith from formula (\ref{formula}).
		
		At this point, $(viii)$ follows directly from the fact that every partition $\pi$ satisfies $\textrm{m}_X\preceq \pi\preceq \textrm{g}_X$ and that $\mathfrak{N}$ is an isotone function.
		
		We now focus on $(ix)$, which requires a little bit more of elaboration. Let $\pi$ and $\pi'$ be arbitrary noncomparable partitions, and let us consider once again an increasing sequence $\pi\wedge\pi'=\pi_0\sqsubset\pi_1\sqsubset\pi_2\sqsubset\dots\sqsubset\pi_l=\pi$. We already know that there are atoms $\pi_{x_{a_i}x_{b_i}}$, $1\leq i\leq l$ such that $\pi_i=\pi_{i-1}\vee \pi_{x_{a_i}x_{b_i}}$.  Thus, 
		$\pi_0\vee\pi_{x_{a_1}x_{b_1}}\vee\pi_{x_{a_2}x_{b_2}}\sqsubset\dots\sqsubset\pi_0\vee\pi_{x_{a_1}x_{b_1}}\vee\pi_{x_{a_2}x_{b_2}}\vee\ldots\vee\pi_{x_{a_l}x_{b_l}}=\pi.$
		Note that for whichever atom $\pi_{x_{a_i}x_{b_i}}$, $\pi_{x_{a_i}x_{b_i}}$ cannot refine $\pi_{i-1}$. If that were the case, $\pi_{i-1}\vee \pi_{x_{a_i}x_{b_i}}$ would be equal to $\pi_{i-1}\neq \pi_{i}$. Consequently, $\pi_{x_{a_i}x_{b_i}}$ is also unable to refine $ \pi\wedge \pi'$. With this in hand, and using the fact that $\pi_{x_{a_i}x_{b_i}}$ refines $\pi$, we can assert that $\pi_{x_{a_i}x_{b_i}}$ is prevented from refining $\pi'$.
		Taking the join of every partition in the sequence above and $\pi'$, we obtain $\pi'=\pi'_0\sqsubset\pi'_1\sqsubseteq\dots\sqsubseteq\pi'_l=\pi\vee \pi'$, with $\pi'_i=\pi_i\vee \pi'$. This sequence can be rewritten as
		$\pi' \sqsubset\pi'\vee\pi_{x_{a_1}x_{b_1}} \sqsubseteq \pi'\vee\pi_{x_{a_1}x_{b_1}}\vee\pi_{x_{a_2}x_{b_2}}\sqsubseteq\dots\sqsubseteq\pi'\vee\pi_{x_{a_1}x_{b_1}}\vee\pi_{x_{a_2}x_{b_2}}\vee\ldots\vee\pi_{x_{a_l}x_{b_l}}=\pi'\vee \pi$. To simplify the notation, let us write $\textrm{s}= \{\pi_{x_{a_1}x_{b_1}},\pi_{x_{a_2}x_{b_2}},\ldots,\pi_{x_{a_l}x_{b_l}}\}$.

		We now come to the main idea. There is a natural mapping that identifies every minimal atomic decomposition $\mathfrak{a}_{\pi\wedge\pi'}$ of $\pi\wedge\pi'$ with the minimal atomic decomposition $\mathfrak{a}_{\pi}:=\mathfrak{a}_{\pi\wedge\pi'}\cup\textrm{s}$ of $\pi$. We can think of this mapping as the composition $\Psi_{x_{a_l}x_{b_l}}\circ\ldots \circ\Psi_{x_{a_2}x_{b_2}}\circ\Psi_{x_{a_1}x_{b_1}}$. A minimal atomic decomposition $\mathfrak{a}_{\pi_0}$ of $\pi_0$ is taken by $\Psi_{x_{a_1}x_{b_1}}$, which returns the atomic decomposition $\mathfrak{a}_{\pi_1}= \mathfrak{a}_{\pi}\cup \{\pi_{x_{a_1}x_{b_1}} \}$. Then $\mathfrak{a}_{\pi_1}$ is taken by $\Psi_{x_{a_2}x_{b_2}}$ which produces $\mathfrak{a}_{\pi_2}$, that is received by $\Psi_{x_{a_3}x_{b_3}}$. The process continue as such. Lemma \ref{addinganatominjections} ensures that every mapping $\Psi_{x_{a_i}x_{b_i}}$ is an injection, and so is their composition (see Corollary \ref{compositionofaddinganatom}).
		
		We can proceed in a similar manner with the minimal atomic decompositions of $\pi'$; however, if at any step in the sequence $\pi'=\pi'_0\sqsubset\pi'_1\sqsubseteq\dots\sqsubseteq\pi'_l=\pi\vee \pi'$ the equality holds, then the corresponding atom will be redundant. Let us suppose $\pi'_{j-1}=\pi'_{j}$. Thus, for whichever minimal atomic decomposition $\mathfrak{a}_{\pi_{j-1}}$ of $\pi_{j-1}$, the operation $\mathfrak{a}_{\pi_{j-1}}\cup \pi_{x_{a_j}x_{b_j}}$ yields a redundant atomic decomposition of $\pi_{j-1}$. Then there are atoms $\pi_{x_{a_j}x_{c_1}}, \pi_{x_{c_1}x_{c_2}}, \ldots, \pi_{x_{c_l}x_{b_j}}$, not all of which lie $\textrm{s}$. From among those atoms lying outside $\textrm{s}$, let us remove that $\pi_{x_{c_k},x_{c_{k+1}}}$ for which $c_k$ is the least (if various atoms share $x_{c_k}$ for the least $c_k$, we also minimize $c_{k+1}$.) Let us denote by $\Delta_j$ the mapping which performs such a remotion. We argue that $\Delta_j\circ\Psi'_{x_{a_j}x_{b_j}}$ is an injection from the set of all minimal atomic decompositions of the partition $\pi'_j$ into itself. Now we identify every minimal atomic decomposition $\mathfrak{a}_{\pi'}$ into the minimal atomic decomposition $\left(\Gamma_{l}\circ\ldots \circ\Gamma_{2}\circ\Gamma_{1}\right) (\mathfrak{a}_{\pi'})$ of $\pi\vee\pi'$, where   
		
		\[
		\Gamma_j = \left\{\begin{array}{lr}
		\Psi'_{x_{a_j}x_{b_j}}, & \text{if } \pi'_{j-1}\sqsubset\pi'_{j}\\
		\Delta_j\circ\Psi'_{x_{a_j}x_{b_j}}, & \text{if } \pi'_{j-1}=\pi'_{j}.\\
		\end{array}\right.
		\]
		
		These identifications are characterized by the following.
		\begin{itemize}
			\item every atomic minimal atomic decomposition $\mathfrak{a}_{\pi\wedge\pi'}$ of $\pi\wedge\pi'$ is assigned a minimal atomic decomposition $\mathfrak{a}_{\pi}$ of $\pi$ satisfying $\textrm{s} \subset\mathfrak{a}_{\pi}$  and every atom in $\mathfrak{a}_{\pi}-\textrm{s}$ refines $\pi\wedge\pi'$.
			\item every minimal atomic decomposition $\mathfrak{a}_{\pi'}$ of $\pi'$ is assigned a
		 minimal atomic decomposition $\mathfrak{a}_{\pi\vee\pi'}$ of $\pi\vee\pi'$ satisfying $\textrm{s} \subset\mathfrak{a}_{\pi\vee\pi'}$  and every atom in $\mathfrak{a}_{\pi\vee\pi'}-\textrm{s}$ refines $\pi'$.
		\end{itemize}
		
		Let $\mathfrak{a}_{\pi}$ be an arbitrary minimal atomic decomposition of $\pi$ that matches no minimal atomic decomposition of $\pi\wedge\pi'$ in the sense described above. Then there is an atom $\pi_{xx'}$ that  neither  belongs to $\textrm{s}$ nor refines $\pi\wedge \pi'$. Since $\pi_{xx'}$ does not refine $\pi\wedge\pi'$, $\pi_{xx'}$ does not refine $ \pi'$ either. Let us consider now a sequence $\pi\sqsubset\pi\vee\pi_{x_{r_1}x_{s_1}}\sqsubset\pi\vee\pi_{x_{r_2}x_{s_2}}\sqsubset\dots\sqsubset\pi\vee\pi_{x_{r_m}x_{s_m}}=\pi\vee\pi'.$ It is now straightforward to see that $\textsc{a}_{\pi\vee\pi'}:=\mathfrak{a}_{\pi}\cup \{\pi_{x_{r_1}x_{s_1}},\pi_{x_{r_2}x_{s_2}},\ldots,\pi_{x_{r_m}x_{s_m}}\}$ is a minimal atomic decomposition of $\pi\vee\pi'$. This decomposition $\textsc{a}_{\pi\vee\pi'}$ matches no minimal atomic decomposition of $\pi'$ because it contains $\pi_{xx'}$. The mapping that sends $\mathfrak{a}_{\pi}$ in $\textsc{a}_{\pi\vee\pi'}$ in an injection. Therefore,
		
		\begin{equation}\label{supermodularity}
		\mathfrak{N}(\pi) -  \mathfrak{N}(\pi\wedge \pi')\leq \mathfrak{N}(\pi\vee \pi')- \mathfrak{N}(\pi').
		\end{equation}
		
		We have thus proven that $\mathfrak{N}$ is a supermodular function.
		
		Property $(x)$ is a consequence of the construction presented in \cite{LeclercPartitions}.
		
	\end{proof}

Let us consider a simple example. Let \( X = \{x_1, x_2, x_3, x_4\} \), and consider the partitions
\(\pi = \{\{x_1,x_2\}, \{x_3,x_4\}\}\) and \(\pi' = \{\{x_1,x_3\}, \{x_2,x_4\}\}\).
We will verify the inequality
\[\mathfrak{N}(\pi \wedge \pi') + \mathfrak{N}(\pi') \le \mathfrak{N}(\pi) + \mathfrak{N}(\pi \vee \pi'),\]
by constructing injections \(\Psi: M(\pi \wedge \pi') \to M(\pi)\) and \(\Gamma_: M(\pi') \to M(\pi \vee \pi')\).

Notice first that
\(\pi \wedge \pi' = \{\{x_1\},\{x_2\},\{x_3\},\{x_4\}\}\) and \(\pi \vee \pi' = \{\{x_1,x_2,x_3,x_4\}\}\). Also,
the atoms of \( \Pi(X) \) are the six partitions with one block of size two:
\[
\begin{aligned}
&\pi_{x_1,x_2} = \{\{x_1,x_2\},\{x_3\},\{x_4\}\}, 
&&\pi_{x_1,x_3} = \{\{x_1,x_3\},\{x_2\},\{x_4\}\}, \\
&\pi_{x_1,x_4} = \{\{x_1,x_4\},\{x_2\},\{x_3\}\}, 
&&\pi_{x_2,x_3} = \{\{x_2,x_3\},\{x_1\},\{x_4\}\}, \\
&\pi_{x_2,x_4} = \{\{x_2,x_4\},\{x_1\},\{x_3\}\}, 
&&\pi_{x_3,x_4} = \{\{x_3,x_4\},\{x_1\},\{x_2\}\}.
\end{aligned}
\]

The respective sets of minimal decompositions for the partitions $\pi\wedge\pi'$, $\pi$, $\pi'$ and $\pi\vee\pi'$ are given below.
\[
\begin{aligned}
M(\pi \wedge \pi') &= \{\varnothing\}, \\[6pt]
M(\pi) &= \bigl\{\{\pi_{x_1,x_2},\,\pi_{x_3,x_4}\}\bigr\}, \\[6pt]
M(\pi') &= \bigl\{\{\pi_{x_1,x_3},\,\pi_{x_2,x_4}\}\bigr\}, \\[6pt]
M(\pi \vee \pi')
&= \Bigl\{ 
\{\pi_{x_1,x_2},\,\pi_{x_1,x_3},\,\pi_{x_1,x_4}\},\;
\{\pi_{x_1,x_2},\,\pi_{x_1,x_3},\,\pi_{x_2,x_4}\},\\
&\quad
\{\pi_{x_1,x_2},\,\pi_{x_1,x_3},\,\pi_{x_3,x_4}\},\;
\{\pi_{x_1,x_2},\,\pi_{x_1,x_4},\,\pi_{x_2,x_3}\},\\
&\quad
\{\pi_{x_1,x_2},\,\pi_{x_1,x_4},\,\pi_{x_3,x_4}\},\;
\{\pi_{x_1,x_2},\,\pi_{x_2,x_3},\,\pi_{x_2,x_4}\},\\
&\quad
\{\pi_{x_1,x_2},\,\pi_{x_2,x_3},\,\pi_{x_3,x_4}\},\;
\{\pi_{x_1,x_2},\,\pi_{x_2,x_4},\,\pi_{x_3,x_4}\},\\
&\quad
\{\pi_{x_1,x_3},\,\pi_{x_1,x_4},\,\pi_{x_2,x_3}\},\;
\{\pi_{x_1,x_3},\,\pi_{x_1,x_4},\,\pi_{x_2,x_4}\},\\
&\quad
\{\pi_{x_1,x_3},\,\pi_{x_2,x_3},\,\pi_{x_2,x_4}\},\;
\{\pi_{x_1,x_3},\,\pi_{x_2,x_3},\,\pi_{x_3,x_4}\},\\
&\quad
\{\pi_{x_1,x_3},\,\pi_{x_2,x_4},\,\pi_{x_3,x_4}\},\;
\{\pi_{x_1,x_4},\,\pi_{x_2,x_3},\,\pi_{x_2,x_4}\},\\
&\quad
\{\pi_{x_1,x_4},\,\pi_{x_2,x_3},\,\pi_{x_3,x_4}\},\;
\{\pi_{x_1,x_4},\,\pi_{x_2,x_4},\,\pi_{x_3,x_4}\}
\Bigr\}.
\end{aligned}
\]

To define \( \Psi_{x_{a_j},x_{b_j}}\), we choose a maximal chain 
\(
\pi_0 = \pi \wedge \pi' < \pi_1 < \pi_2 = \pi,\)
where
\(
\textrm{s}=\{\pi_{x_1,x_2}, \pi_{x_3,x_4}\}\),
\(\pi_1 = \pi_{0} \vee \pi_{x_1,x_2}\) and 
\(\pi_2 = \pi_{1} \vee \pi_{x_3,x_4}
\). Thus,
\[
\Psi_{x_{a_j},x_{b_j}}(\mathfrak{a}_{\pi_{i-1}}) = \mathfrak{a}_{\pi_{i-1}} \cup \{\pi_{x_{a_j},x_{b_j}}\}, 
\quad \text{for } i = 1,2.
\]
To define \( \Gamma_j \), we use the chain
\(
\pi'_0 = \pi' < \pi'_1 = \pi'_2 = \pi \vee \pi',
\)
with
\(
\pi'_0 = \{\{x_1,x_3\},\{x_2,x_4\}\}\), 
\(
\pi_1' = \pi_{0}' \vee \pi_{x_1,x_2}\) and 
\(
\pi_2' = \pi_{1}' \vee \pi_{x_3,x_4}
\).
First, we consider
\[
\Psi'_{x_{a_j},x_{b_j}}(\mathfrak{a}_{\pi'_{i-1}}) = \mathfrak{a}_{\pi'_{i-1}} \cup \{\pi_{x_{a_j},x_{b_j}}\}, 
\quad \text{for } i = 1,2.
\]
Now, according to the proof, at each index \(j\) we define 
\[
\Gamma_{j} =
\begin{cases}
\Psi'_{x_{a_j},x_{b_j}}, 
&\text{if }\pi'_{j-1} < \pi'_{j} \text{ (add atom }\pi_{x_{a_j},x_{b_j}}\!), \\[6pt]
\Psi'_{x_{a_j},x_{b_j}} \circ\Delta_{j}, 
&\text{if }\pi'_{j-1} = \pi'_{j} \text{ (remove one redundant atom).}
\end{cases}
\]
Since \(\pi'_0<\pi'_1\), it follows that 
\[
\Gamma_1 = \Psi'_{x_1,x_2}(\{\pi_{x_1,x_3}, \pi_{x_2,x_4}\}) 
= \{\pi_{x_1,x_3}, \pi_{x_2,x_4}, \pi_{x_1,x_2}\}.
\]
Now, adding \(\pi_{x_3,x_4}\) would create a non-minimal decomposition since \(\pi'_1=\pi'_2\). We apply a reduction map \(\Delta_2\) that removes a redundant atom not in \(s=\{\pi_{x_1,x_2},\pi_{x_3,x_4}\}\); for instance,
\[
\Delta_2(\{\pi_{x_1,x_3}, \pi_{x_2,x_4}, \pi_{x_1,x_2}\}) 
= \{\pi_{x_2,x_4}, \pi_{x_1,x_2}\}.
\]
Finally,
\[\begin{aligned}
\Gamma_2
&= \Psi'_{\pi_{x_3,x_4}}\bigl(\Delta_2(\{\pi_{x_1,x_3}, \pi_{x_2,x_4}, \pi_{x_1,x_2}\})\bigr)\\
&= \Psi'_{\pi_{x_3,x_4}}(\{\pi_{x_2,x_4}, \pi_{x_1,x_2}\})\\
&= \{\pi_{x_2,x_4}, \pi_{x_1,x_2}, \pi_{x_3,x_4}\}
\;\in\;M(\pi \vee \pi').
\end{aligned}\]

Hence, \(\Gamma=\Gamma_2\circ\Gamma_1\) is injective.

Putting it all together:
\[\begin{aligned}
\emptyset &= M(\pi \wedge \pi') 
\xrightarrow{\Psi} 
\bigl\{\{\pi_{x_1,x_2},\pi_{x_3,x_4}\}\bigr\} = M(\pi),\\
\{\{\pi_{x_1,x_3},\pi_{x_2,x_4}\}\} &= M(\pi') 
\xrightarrow{\Gamma} 
\{\pi_{x_2,x_4},\pi_{x_1,x_2},\pi_{x_3,x_4}\} \subset M(\pi \vee \pi').
\end{aligned}\]
Since all elements of \(M(\pi)\) are matched but \(M(\pi \vee \pi')\) has extra elements,
\[
\mathfrak{N}(\pi) - \mathfrak{N}(\pi \wedge \pi') \;\le\; \mathfrak{N}(\pi \vee \pi') - \mathfrak{N}(\pi'); \text{ i.e., }  1-1 < 16-1.
\]
	
	\section{Constrained atomic decompositions}
	
	Let us now consider a subset $\mathcal{R}\subseteq\mathcal{A}$ of atoms, referred to as red atoms. (We refer to this distinguished subset as ``red atoms" to highlight that our analysis focuses on a specific ---albeit arbitrary--- selection of atoms, rather than the entire set of atoms in the lattice.) The purpose of this section is to derive a recursive formula for $\bm{\pi}(X,j,s,\mathcal{R})$ ---the total number of partitions lying at the $j$th level of $\Pi(X)$ which can be expressed as the join of $s$ red atoms.
	
	 To this end, it is useful if we can determine first the number of partitions in $\Pi(X)$ that can be expressed as the join of red atoms. Henceforth, let $\Pi(X,\mathcal{R})$ denote the set of partitions of $X$ that can be generated with red atoms. In other words, the partition $\pi\in\Pi(X,\mathcal{R})$ if and only if there are red atoms $\pi_1,\pi_2,\ldots,\pi_k\in \mathcal{R}$ such that $\pi=\bigvee_{i=1}^k \pi_i$ (regardless the number of red atoms $k$ and the level at which $\pi$ lies).
	
	Let $\mathcal{S}(G_{\mathcal{R}})$ denote the set all whose elements are the spanning subgraphs of $G_{\mathcal{R}}$, and let us consider the equivalence relation $\sim$ on $\mathcal{S}(G_{\mathcal{R}})$ given by: for all $G,G'\in\mathcal{S}(G_{\mathcal{R}})$, $G\sim G'$ if and only if $\pi_{G}=\pi_{G'}$.	
	
	\begin{Proposition}\label{graphsequivalenceclassesgivepartitions}
		$\Pi(X,\mathcal{R})$ is equipotent to the quotient set $\mathcal{S}(G_{\mathcal{R}})/\sim$.
	\end{Proposition}
	\begin{proof}
		Let $G$ be a spanning subgraph graph of $G_{\mathcal{R}}=(X, E_{\mathcal{R}})$ and consider the partition $\pi_{G}$. Since the edges of $G_{\mathcal{R}}=(X, E_{\mathcal{R}})$ correspond to the atom in $\mathcal{R}$, $\pi_{G}$ can be expressed as the join of red atom. Let us consider now the map $\lambda: \mathcal{S}(G_{\mathcal{R}})\rightarrow \Pi(X,\mathcal{R})$ that to a spanning graph $G$ of $G_{\mathcal{R}}=(X, E_{\mathcal{R}})$ assigns the partition $\pi_{G}$.  We argue that $\lambda$ is surjective. Indeed, take any partition $\pi\in\Pi(X,\mathcal{R})$ and define the spanning subgraph $G$ of $G_{\mathcal{R}}=(X, E_{\mathcal{R}})$ whose edges correspond to those red atoms that refine $\pi$. Since we can recover $\pi$ by taking the join of such red atom, we can conclude $\lambda(G)=\pi$. The results now follows from the first isomorphism theorem.
	\end{proof}
	
	Now, we shall focus on computing the number of equivalence classes in $\mathcal{S}(G_{\mathcal{R}})/\sim$.
	
	\begin{Lemma}\label{forest}
		If $G_{\mathcal{R}}$ is a forest, then the quotient set $\mathcal{S}(G_{\mathcal{R}})/\sim$ consists of exactly $2^{|E_{\mathcal{R}}|}$ equivalence classes.
	\end{Lemma}
	\begin{proof}
		We will start by proving that if $G_1,G_2\in\mathcal{S}(G_{\mathcal{R}})$ with edge sets $E_1$ and $E_2$, respectively, are spanning subgraphs of $G_{\mathcal{R}}$ for which there exists an edge $\{x_i, x_j\} \in E_{\mathcal{R}}$ satisfying $\{x_i, x_j\} \in E_1$ and $ \{x_i, x_j\}\notin E_2$, then $G_1$ and $G_2$ are not equivalent. 
		
		Indeed, since $G_{\mathcal{R}}$ is a forest, the only path connecting $x_i$ to $x_j$ in $G_{\mathcal{R}}$ is the edge $\{x_i, x_j\}$. This implies that if any spanning subgraph of $G_{\mathcal{R}}$ contains a path connecting $x_i$ to $x_j$, that path must be the edge  $\{x_i, x_j\}$. Therefore, there is a path in $G_1$ connecting $x_i$ to $x_j$ but $G_2$ fails to contain such a path. This means that $x_i$ and $x_j$ are in the same connected component of $G_1$ but in different connected components of $G_2$, which proves that the respective induced partitions $\pi_{G_1}$ and $\pi_{G_2}$ are different. Consequently, $G_1$ and $G_2$ are not equivalent. Thus, every spanning subgraph is just equivalent to itself, which means that the quotient set $\mathcal{S}(G_{\mathcal{R}})/\sim$ has as many equivalent classes as spanning subgraphs $G_{\mathcal{R}}$. It follows then that such a number is $2^{|E_{\mathcal{R}}|}$.
	\end{proof}
	
	\begin{Lemma}\label{cycle1}
		If $G_{\mathcal{R}}$ is a (simple) cycle, then, for any $G_1, G_2\in\mathcal{S}(G_{\mathcal{R}})$,  if $G_1\neq G_2$, and $|E_1| \leq |E_{\mathcal{R}}| - 2$ and $|E_2| \leq |E_{\mathcal{R}}| - 2$, then $G_1$ and $G_2$ are not equivalent.
	\end{Lemma}
	\begin{proof}
		We shall divide the proof in two cases.
		
		\emph{Case 1:} $E:=(E_{\mathcal{R}}-E_1)\cap (E_{\mathcal{R}}-E_2)\neq\emptyset$. Then the graph $\hat{G}_{\mathcal{R}}=(X, E_{\mathcal{R}}-E)$ is a forest, and $G_1$ and $G_2$ are also spanning subgraphs of $\hat{G}_{\mathcal{R}}$. Combining Lemma \ref{forest} with the fact that $G_1\neq G_2$, we conclude that $G_1$ and $G_2$ are not equivalent. 
		
		\emph{Case 2:} $E=\emptyset$. Let us suppose that $e_1=\{x,x'\}$ and $e_2 $ are two of the edges that were removed from $G_{\mathcal{R}}$ to obtain $G_1$, i.e., $e_1,e_2\notin E_1$. Then there is no path in $G_1$ connecting $x$ and $x'$. However, since $E=\emptyset$, $e_1\in E_2$ and hence there is a path in $G_1$ connecting $x$ and $x'$.
		This proves that $G_1$ and $G_2$ are not equivalent.
	\end{proof}
	
	The next result is obvious.
	
	\begin{Lemma}\label{cycle2}
		If $G_{\mathcal{R}}$ is a cycle, then every spanning subgraph $G=(X,E)$ of $G_{\mathcal{R}}$ such that $|E|\geq|E_{\mathcal{R}}|-1$  is equivalent to $G_{\mathcal{R}}$.
	\end{Lemma}
	
	Combining Lemmas \ref{cycle1} and \ref{cycle2}, we get the following result.
	\begin{Proposition}\label{cycles}
		If $G_{\mathcal{R}}$ is a cycle, then the quotient set $\mathcal{S}(G_{\mathcal{R}})/\sim$ consists of exactly $2^{|E_{\mathcal{R}}|}-|E_{\mathcal{R}}|$ equivalence classes.
	\end{Proposition}
	
	\begin{proof}
		From Lemma \ref{cycle2}, we know that the equivalence class of $G_{\mathcal{R}}$ consists of all those spanning subgraphs of $G_{\mathcal{R}}$ obtained from $G_{\mathcal{R}}$ by removing at most one edge are equivalent, which add up $|E_{\mathcal{R}}|+1$. From Lemma \ref{cycle1}, every spanning subgraph of $G_{\mathcal{R}}$ obtained from $G_{\mathcal{R}}$ by removing more than one is only equivalent to itself. Since there is a total of $2^{|E_{\mathcal{R}}|}$ spanning subgraphs, there are $2^{|E_{\mathcal{R}}|}-|E_{\mathcal{R}}|-1$ of such equivalence classes. 
		Therefore, total number of equivalence classes is $2^{|E_{\mathcal{R}}|}-|E_{\mathcal{R}}|$. 
	\end{proof}
	
	\begin{Theorem}
		Let \( G_1, G_2, \dots, G_k\in \mathcal{S}(G_{\mathcal{R}})\), $G_i=(X,E_i)$, be pairwise edge-disjoint spanning subgraphs of \(G_{\mathcal{R}} \) such that:
		\begin{enumerate}
			\item  \( \sum_{i=1}^{k} G_i = G_{\mathcal{R}} \), i.e., $\bigcup_{i=1}^k E_i=E$. (Here, by the sum of two graphs $G=(V,E)$ and $H=(V',E')$ we mean the graph $G+H=(V\cup V', E\cup E')$.)
			\item  For every cycle $C$ in $G_{\mathcal{R}}$, there is $i\in\{1,2,\ldots,k\}$ such that $C$ is a subgraph of $G_i$.
		\end{enumerate}
		Then, $|\mathcal{S}(G_{\mathcal{R}})/\sim| = \prod_{i=1}^{k} |\mathcal{S}(G_i)/\sim|$
	\end{Theorem}
	
	\begin{proof}
		Note that if $G'$ is a spanning subgraph of $G_{\mathcal{R}}$, then our hypothesis ensures that we can find $G'_1, \ldots, G'_k$ such that, for every $i \in\{ 1, 2,\ldots, k\}$,   $G'_i \subseteq G_i$, such that $\sum_{i=1}^{k} G'_i = G' $.
		
		Let us consider now two arbitrary spanning subgraphs $G'$ and $\hat{G}$ of $G_{\mathcal{R}}$, and let  \( \{G'_1, G'_2, \ldots, G'_k\} \) and \( \{\hat{G}_1, \hat{G}_2, \ldots, \hat{G}_k\} \) be their respective decompositions satisfying the condition above, i.e., for every \( i\in\{ 1, 2, \ldots, k\}\), \( G'_i \subseteq G_i \), \( \hat{G}_i \subseteq G_i \), and $\sum_{i=1}^{k} G'_i = G' $, $\sum_{i=1}^{k} \hat{G}_i = \hat{G}$. We argue that if there is an index \( j \in\{ 1, 2, \ldots, k\}\) such that \( G'_j \) are \( \hat{G}_j \) are not equivalent, then $G'$ and $\hat{G}$ are not equivalent either. Indeed, if \( G'_j \) are \( \hat{G}_j \) are not equivalent, then there are elements $x$ and $x'$ in $X$ which are connected by a path in one of these graphs but not in the other. Without loss of generality, let us suppose that $G'_j$ is the subgraph containing a path that connects $x$ and $x'$. This implies, that $x$ and $x'$ are connected by path in $G'$. Now, notice that $x$ and $x'$ cannot be connected by a path in $\hat{G}$ because such a path must be contained in one of the $\hat{G}_u$, for some $u\neq j$. But this either contradicts the fact that $\hat{G}_u\subset G_u$ which is edge-disjoint with $G_j$, so cannot contain the same path we used for $G'_j$, or, if it were a different path, the fact that every cycle of $G_{\mathcal{R}}$ is contained in one of the $G_i$'s.
		Therefore, $G'$ and $\hat{G}$ are not equivalent. The desired formula now follows immediately.
	\end{proof}
	
	\begin{Corollary} \label{corollaryedgedisjointcycles}
		If all the cycles $C_1, C_2 \ldots, C_k$ of \(G_{\mathcal{R}} \) are pairwise edge-disjoint, and $T$ is the spanning subgraph all whose edges are those that lie in no cycle of \(G_{\mathcal{R}} \), then \begin{equation}\label{disjointcycles}
		|\mathcal{S}(G_{\mathcal{R}})/\sim| = |\mathcal{S}(T)/\sim|\cdot\left|\prod_{i=1}^{k} |\mathcal{S}(C_i)/\sim\right|\end{equation} Moreover, by virtue of Lemma \ref{forest} and Proposition \ref{cycles}, Formula (\ref{disjointcycles}) can be rewritten as:
		\begin{equation}	|\mathcal{S}(G_{\mathcal{R}})/\sim|= 2^{\left(|\mathcal{R}| -\sum_{i=1}^{k} |E_i|\right)}  \cdot \prod_{i=1}^{k} (2^{|E_i|}-|E_i|), \end{equation} where $E_i$ denotes the edge set of the cycle $C_i$.
	\end{Corollary}
	
	Corollary \ref{corollaryedgedisjointcycles} can be generalized as follows. Consider the equivalence relation on the set of all cycles of \(G_{\mathcal{R}} \) given by: Two cycles, $C$ and $C'$, are equivalent if there exists a sequence of cycles \(C, C_1, C_2, \ldots, C_m, C'\) such that any two consecutive cycles in the sequence have common edges (at least one). Henceforth, $[C]$ denotes the equivalence class of the cycle $C$ and $C^{\propto}$ denote the sum of the cycles equivalent to the class representative $C$.
	
	\begin{Corollary}
		Let $C_1, C_2 \ldots, C_k$ be all the cycles of \(G_{\mathcal{R}} \) and let $T$ be the spanning subgraph all whose edges are those that lie in no cycle of \(G_{\mathcal{R}} \), then \begin{equation}\label{anycycles}
		|\mathcal{S}(G_{\mathcal{R}})/\sim| = |\mathcal{S}(T)/\sim|\cdot\left|\prod_{i=1}^{k} |\mathcal{S}(C^{\propto}_i)/\sim\right|\end{equation}
	\end{Corollary}
	
	Given a collection $\mathcal{F}\subseteq \mathcal{S}(G_{\mathcal{R}})$ of spanning subgraphs of $G_{\mathcal{R}}$ and $e=\{x,x'\}\in E_{\mathcal{R}}$, i.e., $\pi_{xx'}\in \mathcal{R}$, we shall consider the following subcollections:
	
	\begin{eqnarray*}
		\mathcal{F}_e&:=&\{G=(X,E)\in\mathcal{F}\,:\, e\in E\};\\
		\mathcal{F}_{x|x'} &:=&\{G\in\mathcal{F}\,:\, x \text{ and } x' \text{ lie in different connected components of } G\}.
	\end{eqnarray*}
	
	\begin{Proposition}\label{RecursiveFormula} 	
		Let	$\mathcal{F}$ be a collection of spanning subgraphs of $G_{\mathcal{R}}$ and let $e\in E_{\mathcal{R}}$ be an edge for which $\mathcal{F}_e\neq\emptyset$. If for every graph $G\in\mathcal{F}$ containing a path $p$ connecting $x$ and $x'$ there is a graph $\hat{G}\in\mathcal{F}_e$ equivalent to $G$, then
		\begin{equation}
		|\mathcal{F}/\sim| = 	|\mathcal{F}_e/\sim|+ 	|\mathcal{F}_{x|x'}/\sim|.
		\end{equation}	
	\end{Proposition}
	\begin{proof}
		Notice first that for any graph $G'\in\mathcal{F}_{x|x'}$, the corresponding partition $\pi_{G'}$ places $x$ and $x'$ in different blocks, while the partition $\pi_{\hat{G}}$ corresponding to a graph $\hat{G}\in\mathcal{F}_e$ places $x$ and $x'$ in the same block. Thus, the every graph $G'\in\mathcal{F}_{x|x'}$ can only be equivalent to graphs lying in $\mathcal{F}_{x|x'}$. This implies that the equivalence class of a $G'\in\mathcal{F}_{x|x'}$ in the entire set $\mathcal{F}$ is the same as its equivalence class in the subset $\mathcal{F}_{x|x'}$. On the other hand, every graph $G\in\mathcal{F}$ containing a path $p$ that connects $x$ and $x'$ is equivalent to some graph $\hat{G}\in\mathcal{F}_e$, the remaining classes admit a representative from $\mathcal{F}_e$. This proves the result.
	\end{proof}
	
	\begin{Corollary}
		$\mathcal{S}(G_{\mathcal{R}})$ satisfies the premises of Proposition \ref{RecursiveFormula} for every edge $e=\{x,x'\}\in E_{\mathcal{R}}$. Therefore,
		\begin{equation}
		|\mathcal{S}(G_{\mathcal{R}})/\sim| = 	|\mathcal{S}(G_{\mathcal{R}})_e/\sim|+ 	|\mathcal{S}(G_{\mathcal{R}})_{x|x'}/\sim|.
		\end{equation}
	\end{Corollary}
	
	By virtue of Proposition \ref{graphsequivalenceclassesgivepartitions}, the previous results can be rephrased in term of atomic decomposition. We shall focus on Proposition \ref{RecursiveFormula}. 
	
	Given a red atom $\pi_{xx'}\in\mathcal{R}$ and a collection $\mathcal{J}\subseteq \left(2^{\mathcal{R}}-\{\emptyset\}\right)$ of nonempty subsets of $\mathcal{R}$, we shall consider the following subcollections:
	
	\begin{eqnarray*}
		\mathcal{J}_{xx'}&:=&\{J \in\mathcal{J}\,:\, \pi_{xx'}\in J\};\\
		\mathcal{J}_{x|x'} &:=&\{J\in\mathcal{J}\,:\, \pi_{xx'}\text{ does not refine }\bigvee J\};\\
		\Pi(X,\mathcal{J})&:=&\{\pi\in\Pi(X)\,:\, \pi=\bigvee J \text{ for some } J\in\mathcal{J}\}
	\end{eqnarray*}
	
Note that $\Pi(X,\mathcal{R})= \Pi(X,\mathcal{J})$ for $\mathcal{J}= 2^{\mathcal{R}}-\{\emptyset\}$.
	
	\begin{Proposition}\label{RecursiveFormulaAtoms} 	
		Let	$\mathcal{J}$ be a collection of nonempty subsets of $\mathcal{R}$ and let $\pi_{xx'}\in \mathcal{R}$ be a red atom for which $\mathcal{J}_{xx'}\neq\emptyset$. If for every $J\in\mathcal{J}$ such that $\pi_{xx'}\preceq\bigvee J$ there is $\hat{J}\in\mathcal{J}_{xx'}$ such that $\bigvee J=\bigvee\hat{J}$, then
		\begin{equation}
		|\Pi(X,\mathcal{J})| = 	|\Pi(X,\mathcal{J}_{xx'})|+ 	|\Pi(X,\mathcal{J}_{x|x'})|.
		\end{equation}	
	\end{Proposition}
	\begin{proof} Notice that if $\pi\in \Pi(X,\mathcal{J})$, then by definition  $\pi=\bigvee J$ for some $J\in\mathcal{J}$. If $\pi_{xx'}\preceq\pi$, by hypothesis, there is $\hat{J}\in\mathcal{J}_{xx'}$ such that $\pi=\bigvee J =\bigvee \hat{J}$. Thus, $\pi\in \Pi(X,\mathcal{J})$ that is refined by the atom $\pi_{xx'}$ actually belongs to $\Pi(X,\mathcal{J}_{xx'})$. On the other hand, if $\pi\in \Pi(X,\mathcal{J})$ is not refined by $\pi_{xx'}$, then $\pi\in\Pi(X,\mathcal{J}_{x|x'})$. Therefore, the proposition follows.
	\end{proof}

	Henceforth, we adopt a more detailed notation for certain subsets of $\Pi(X,\mathcal{R})$. We denote by $\Pi(X,\ast, s, \mathcal{R})$ the set of all partitions of $X$ that can be expressed as the join of exactly $s$ atoms in $\mathcal{R}$; that is  $\Pi(X,\ast, s, \mathcal{R})=\Pi(X,\mathcal{J})$, for $\mathcal{J}=\{J\subseteq \mathcal{R}\,:\, |J|=s\}$). Here, the symbol $\ast$ indicates that the level (or rank) at which the partition lie is arbitrary and unspecified. We further define $\Pi(X,j, s, \mathcal{R})$  as the subset of $\Pi(X,\ast, s, \mathcal{R})$ consisting of those partitions lying at the $j$th level of $\Pi(X)$; that is, partitions of rank $j$. 
	It is important to observe that the fact that $\pi\in\Pi(X,\ast, s, \mathcal{R})$ does not preclude $\pi$ from also belonging to $\Pi(X,\ast, v, \mathcal{R})$, for some $v\neq s$.  Analogously, we define the sets $\Pi(X, j, s, \mathcal{R})_{xx}$ and $\Pi(X, j, s,\mathcal{R})_{x|x}$. 
	
	\begin{Corollary}\label{CorollaryProposition4.4} Let $\pi_{xx'}\in \mathcal{R}$. Then,
		\begin{eqnarray}
		|\Pi(X, \ast, s, \mathcal{R})| &= &|\Pi(X, \ast, s, \mathcal{R})_{xx}|+|\Pi(X,\ast, s, \mathcal{R})_{x|x}|.\\
		|\Pi(X, j, s,\mathcal{R})|& = & |\Pi(X, j, s,\mathcal{R})_{xx}|+|\Pi(X, j, s,\mathcal{R})_{x|x}|.\\
		\nonumber
		\end{eqnarray}
	\end{Corollary}
	
	\begin{proof}
		Let $J$ be an arbitrary size-$s$ subset of $\mathcal{R}$. We claim that, for every atom $\pi_{xx'}\in \mathcal{R}$, if $\pi_{xx'}\preceq\bigvee J$, then there is $\hat{J}\subseteq \mathcal{R}$ such that $|\hat{J}|=s$, $\pi_{xx'}\in \hat{J}$, and $\bigvee J=\bigvee\hat{J}$. Indeed, if $\pi_{xx'}\in J$, we may simply take  $\hat{J}:=J$. Otherwise, let $G_J$ be the graph whose vertex set is $X$ and whose edges correspond to the atoms in $J$. Since $\pi_{xx'}\preceq\bigvee J$, there is a path $p$ in $G_J$ connecting $x$ and $x'$. Consider the graph $G_{J'}$ obtained from $G$ by adding the edge $e=\{x,x'\}$; that is, setting $J'=J\cup\{e\}$.  Observe that $G_{J'}$ has exactly one more edge than $G_J$ and contains the cycle formed by the path $p$ together with the edge $e$. Let $e'\neq e$ be another edge of this cycle. Then, setting $\hat{J}=J'-\{e'\}$, the graph $G_{\hat{J}}$ obtained from $G_{J'}$ by removing the edge $e'$ has as many edges as $G_J$. In addition, since $G_{\hat{J}}$ maintains the same connectivity as $G_J$,  so that $\pi_{G_J}=\pi_{G_{\hat{J}}}$. Letting $\hat{J}$ be the subset of $\mathcal{R}$ corresponding to the edges of $G_{\hat{J}}$, we conclude that $\bigvee J=\bigvee\hat{J}$. The result now follows from Proposition \ref{RecursiveFormulaAtoms}.
	\end{proof}
	
	We shall conclude this section by establishing a recursive formula for the  $\bm{\pi}(X,j,s,\mathcal{R}):=|\Pi(X,j,s,\mathcal{R})|$. Here, recursive formula refers to expressing $\bm{\pi}(X,j,s,\mathcal{R})$  as linear combination of instances of the same formula involving smaller values for the parameters, where at least one of the following reductions occurs: (1) the set $X$ is reduced; (2) the set $\mathcal{R}$ is reduced; (3) the number if $j$ is reduced; and (4) the number $s$ is reduced. 
	
	Let us assume that $\pi_{xx'}\in\mathcal{R}$ is an arbitrary red atom. 
	If, in addition to the edge $e=\{x,x'\}$, there exist  distinct paths $p_1,p_2,\ldots,p_t$ in $G_{\mathcal{R}}$, each connecting $x$ and $x'$ with respective lengths $l_1\leq l_2\leq\ldots\leq l_t$, then for a $t$-dimensional multi-index $\alpha=(\alpha_1,\alpha_2,\ldots,\alpha_t)$, where $1\leq \alpha_v\leq l_v$ for each $1\leq v\leq t$, we construct the set $U_{\alpha}$ by selecting, from each path $p_v$, the edge at position $\alpha_v$. The goal is to construct a minimal cut $U_{\alpha}$ in which every path in interrupted exactly once. If an edge is removed from a path, and that same edge belongs to another path under consideration, then the latter path must not be interrupted again. Thus, $U_{\alpha}$ contains at most $t$ different edges. 
	
	Now, if $\alpha$ is a $ut$-dimensional multi-index, that is, $\alpha=(\beta_1,\beta_2,\ldots,\beta_u)$, where each $\beta_z$, $1\leq z\leq u$, is a $t$-dimensional multi-index, then we define $U_{\alpha}=\bigcup_{z}U_{\beta_z}$. W will denote the dimension of any multi-index $\alpha$ by $d(\alpha)$.
	
	Furthermore, given $W\subseteq \mathbf{t}$, a subset of path indexes,  define $G^W_{\mathcal{R}}$ as the graph obtained from $G_{\mathcal{R}}$ by collapsing the entire union $\bigcup_{w\in W}p_w$ of the paths $p_w$, $w\in W$, into a single new vertex $x_W$. That is, all vertices and all edges that belong to any of the path $p_w$, $w\in W$, are identified with the vertex $x_W$.  Every edge $\{a,b\}$ in $G_{\mathcal{R}}$ such that $b$ is a vertex of $\bigcup_{w\in W}p_w$, is replaced in $G^W_{\mathcal{R}}$ by the edge $\{a,x_w\}$, redirecting all connections to the collapsed component. Every edge in $G_{\mathcal{R}}$ that is disjoint from the union of paths $\bigcup_{w\in W}p_w$ is retained unchanged in $G^W_{\mathcal{R}}$.
	
	Let $X_W$ denote the vertex set of $G_{\mathcal{R}}^W$, and let $\mathcal{R}_W$ be the set of atoms of $\Pi(X_W)$ that correspond to the edges of $G^W_{\mathcal{R}}$. Additionally, we denote by $\pi_W$ the join of the atom of $\Pi(X)$ that correspond to the edges of $\bigcup_{w\in W}p_w$, and by $s_W$ the total number of such edges.
	
	\begin{Theorem}  $\bm{\pi}(X, 0,0,\mathcal{R})=1$, and for every $j>s$, $\bm{\pi}(X, j,s,\mathcal{R})=0$. Additionally, for $j\leq s$ and any atom $\pi_{xx'}\in\mathcal{R}$, we have:
		
		\begin{enumerate}
			\item If the only path in $G_{\mathcal{R}}$ connecting $x$ and $x'$ is the edge $e=\{x,x'\}$, then
			\begin{equation}
			 \bm{\pi}(X,j,s,\mathcal{R})=\bm{\pi}(X,j-1,s-1,\mathcal{R}')+ \bm{\pi}(X,j,s,\mathcal{R}'), 
			\end{equation}
			where $\mathcal{R}':=\mathcal{R}-\{\pi_{xx'}\}$.
			\item If, in addition to the edge $e=\{x,x'\}$, there are  different paths $p_1,p_2,\ldots,p_t$ whose lengths are $l_1\leq l_2\leq\ldots\leq l_t$, respectively, connecting $x$ and $x'$ in $G_{\mathcal{R}}$, then
			\begin{enumerate}
				\item $\bm{\pi}(X,j,j,\mathcal{R})=\bm{\pi}(X,j-1,j-1,\mathcal{R}')+ \bm{\pi}(X,j,j,\mathcal{R}')$, for $j=s\leq l_1$.\\
				\item For $j= s> l_1$,
				\begin{align*}
				\bm{\pi}(X, j, j, \mathcal{R}) = \sum_{u=1}^{\prod_{v=1}^t l_v} \sum_{\substack{\alpha \\ d(\alpha) = ut}} (-1)^{u+1} \Big[ 
				& \bm{\pi}(X, j-1, j-1, \mathcal{R}' - U_{\alpha}) \\
				& + \bm{\pi}(X, j, j, \mathcal{R}' - U_{\alpha}) 
				\Big].
				\end{align*} 
				
				\item For $j= s> l_1$, \begin{align*}
				\bm{\pi}(X, j, s, \mathcal{R}) = 
				& \sum_{\emptyset \neq W \subseteq \mathbf{t}} \sum_u \sum_{\substack{\gamma \\ d(\gamma) = u(t - |W|)}} (-1)^{u+1} 
				\bm{\pi}(X_W, j_W, s - s_W, \mathcal{R} - U_{\gamma}) \\
				+\, & \sum_u \sum_{\substack{\gamma \\ d(\gamma) = u(t - 1)}} (-1)^{u+1} 
				\bm{\pi}(X_e, j - 1, s - 1, \mathcal{R}' - U_{\gamma}) \\
				+\, & \sum_u \sum_{\substack{\alpha \\ d(\alpha) = ut}} (-1)^{u+1} 
				\bm{\pi}(X, j, s, \mathcal{R}' - U_{\alpha}),
				\end{align*}
				where $j_W=n_W-(n-j)$.
					\end{enumerate}
		\end{enumerate}
		
	\end{Theorem}
	\begin{proof}
		The join of $s$ atoms lies at most at the $s$th level of $\Pi(X)$. Therefore, $\Pi(X,j,s,\mathcal{R})=\emptyset$ for every $j>s$, and consequently, $\bm{\pi}(X, j,s,\mathcal{R})=0$.
		
		According to Corollary \ref{CorollaryProposition4.4},   	\begin{equation}
		|\Pi(X, j, s,\mathcal{R})| = |\Pi(X, j, s,\mathcal{R})_{xx}|+|\Pi(X,j, s,\mathcal{R})_{x|x}|.
		\end{equation}
		
		 Let us suppose that $j\leq s$. If $e$ is the only path connecting the elements $x$ and $x'$, then $e$ is a cut-edge and hence the sets $\Pi(X,j,s,\mathcal{R})_{xx'}$ and $\Pi(X,j-1,s-1,\mathcal{R})_{x|x'}$ are equipotent, and their common cardinality is the same as that of the set $\Pi(j-1,s-1,\mathcal{R}')$, namely $\bm{\pi}(j-1,s-1,\mathcal{R}')$. Indeed, if $\pi\in\Pi(j-1,s-1,\mathcal{R}')$, then $\pi\vee\pi_{xx'}\in \Pi(j,s,\mathcal{R})_{xx'}$. Moreover, if $\pi,\pi'\in\Pi(j-1,s-1,\mathcal{R}')$ and $\pi\neq\pi'$, then $\pi\vee\pi_{xx'}\neq\pi'\vee\pi_{xx'}$ given that $x$ and $x'$ belong to different connected components of $G_{\mathcal{R}'}$. On the other side, since for every $\pi\in\Pi(j,s,\mathcal{R})_{x|x'}$, $\pi_{xx'}$ fails to refine $\pi$, the atom $\pi_{xx'}$ is never part of the $s$ red atoms used to decompose $\pi$ as the join of $s$ red atoms. Hence, the cardinality of $\Pi(j,s,\mathcal{R})_{x|x'}$ is the same as that of $\Pi(j,s,\mathcal{R}')$. This proves 1.
		
		However, notice that in the case that there are other paths in $G_{\mathcal{R}}$ connecting $x$ and $x'$, if $s\leq l_1$,  with $s-1$ atoms it is not possible to complete a path other than $e$ connecting $x$ and $x'$. This means that $e$ is still a cut-edge in every subgraph of $G_{\mathcal{R}}$ corresponding to the chosen size-$s$ subset of red atoms containing $e$. In this case, the same equality holds, which proves 2.(a).
		
		To prove 2.(b), notice that if $s\geq l_1$, then, in addition to $e$,  there is at least another path connecting $x$ and $x'$. To ensure that cycles are not formed in the corresponding graph, we need to interrupt every path different from $e$ connecting $x$ and $x'$. For each  $t$-dimensional multi-index $\alpha$, let us consider $\Pi(X,j-1,j-1,\mathcal{R}'-U_{\alpha})$, where $U_{\alpha}$ is as defined above. Notice that $U_{\alpha}$ is a minimum cut that separates $x$ and $x'$. Thus, $\Pi(X,j,j,\mathcal{R})_{xx'}=\Pi(X,j-1,j-1,\mathcal{R}')_{x|x'}=\bigcup_{\alpha}\Pi(j-1,j-1,\mathcal{R}'-U_{\alpha})$. Indeed, if $\pi\in \Pi(X,j,j,\mathcal{R})_{xx'}$, then Proposition \ref{RecursiveFormulaAtoms} ensures that $\pi$ is the join of a size-$j$ subset $\hat{J}$ of $\mathcal{R}$ that contains $\pi_{xx'}$. Since $\pi$ has rank $j$, the corresponding graph $G_{\hat{J}}$ does not contain cycles, and therefore $e$ is a cutting edge of this email. Removing $e$ results into a partition $\hat{\pi}\in \Pi(X,j-1,j-1,\mathcal{R}')_{x|x'}$. Obviously, if  $\pi,\pi'\in \Pi(X,j,j,\mathcal{R})_{xx'}$ and $\pi\neq\pi'$, then the corresponding partitions $\hat{\pi}$ and $\hat{\pi'}$ are different. Notice also that the absence of cycles in $G_{\hat{J}}$ means that all paths connecting $x$ and $x'$ have been interrupted, and hence $\hat{\pi}\in\Pi(X,j-1,j-1,\mathcal{R}'-U_{\alpha})$, for some minimum cut $U_{\alpha}$. Analogously, we can conclude that $\Pi(X,j,j,\mathcal{R})_{x|x'}=\bigcup_{\alpha}\Pi(X,j,j,\mathcal{R}'-U_{\alpha})$.
		Statement 2.(b) now follows from the Inclusion-Exclusion principle to each of the union above.
		
		Note first that if \( \pi \in \Pi(X, j, s, \mathcal{R})_{xx'} \), with \( s > j \), then
		\[
		\pi = \bigvee J
		\]
		for some subset \( J \subseteq \mathcal{R} \) such that \( \pi_{xx'} \in J \) and \( |J| = s \).
		
		Let \( G_{\mathbf{t}} = (X, E_{\mathbf{t}}) \) be the graph defined as follows: an edge \( \hat{e} \in E_{\mathbf{t}} \) if and only if \( \hat{e} = e \) or there exists \( v \in \mathbf{t} \) such that \( \hat{e} \) is an edge of the path \( p_v \).
		
		Among all subsets \( J \subseteq \mathcal{R} \) satisfying \( \pi = \bigvee J \), \( |J| = s \), and \( \pi_{xx'} \in J \), let us denote by \( J_{\pi} \) the subset that contains the greatest number of cycles in \( G_{\mathbf{t}} \). If more than one such subset exists, we choose among them the one(s) containing cycles with the smallest index values.
		
		Let \( C_{xx'} \) be the set of all partitions \( \pi \in \Pi(X, j, s, \mathcal{R})_{xx'} \) such that the corresponding set \( J_{\pi} \) contains at least one cycle that passes through both \( x \) and \( x' \).
		
		For each \( \pi \in C_{xx'} \), define \( p(\pi) \subseteq \mathbf{t} \) as the set of indices of the paths that are common to both \( G_{J_{\pi}} \) and \( G_{\mathbf{t}} \). That is, \( p(\pi) \) records which indexed paths contribute edges to the subgraph induced by \( J_{\pi} \).
		
		For every nonempty subset \( W \subseteq \mathbf{t} \), define the set
		\[
		H_W := \{ \pi \in C_{xx'} \mid p(\pi) = W \}.
		\]
		Note that for all \( W, W' \subseteq \mathbf{t} \), with \( W \neq W' \), we have \( H_W \cap H_{W'} = \emptyset \), so these sets form a disjoint partition of \( C_{xx'} \).
		
		Therefore, we obtain the identity:
		\[
		|C_{xx'}| = \sum_{\emptyset \neq W \subseteq \mathbf{t}} |H_W|.
		\]
		
		To compute the cardinality of \( H_W \), note that in every partition \( \pi \in H_W \), all vertices in the union \( \bigcup_{w \in W} p_w \) lie in the same block. Therefore, \( H_W \) is in bijection with the set of partitions of \( X_W \) into \( n - j \) blocks such that the associated graph contains no cycle passing through the collapsed vertex \( x_W \).
		
		Let \( n_W = |X_W| \) denote the number of vertices in the graph \( G_{\mathcal{R}}^W \). The corresponding partitions lie at level \( n_W - (n - j) \) in the lattice \( \Pi(X_W) \).
		
		To interrupt each of the \( t - |W| \) cycles in \( G_{\mathcal{R}}^W \) that pass through \( x_W \), we proceed as described above: construct a minimal cut \( U_{\gamma} \) by removing exactly one edge from each cycle.
		
		Thus, by the Inclusion-Exclusion principle, we obtain:
		\[
		|H_W| = \sum_u \sum_{\substack{\gamma \\ d(\gamma) = u(t - |W|)}} (-1)^{u+1} \bm{\pi}(X_W, n_W - (n - j), s - s_W, \mathcal{R} - U_{\gamma}).
		\]
		Let \( C_e \) be the set of partitions \( \pi \in \Pi(X, j, s, \mathcal{R})_{xx'} \) for which the \emph{only} path in \( G_{J_{\pi}} \) connecting \( x \) and \( x' \) is the edge \( e \).
		
		We claim that \( \pi \in C_e \) if and only if for every \( v \in \mathbf{t} \),
		\[
		|p_v - \mathcal{R}_{\pi}| \geq 2,
		\]
		where \( \mathcal{R}_{\pi} := \mathcal{A}(\pi) \cap \mathcal{R} \) denotes the set of atoms used in the decomposition of \( \pi \) that belong to \( \mathcal{R} \).
		
		Indeed, suppose \( |p_v - \mathcal{R}_{\pi}| = 0 \) for some \( v \in \mathbf{t} \). Then \( p_v \subseteq J_{\pi} \), and so \( G_{J_{\pi}} \) contains a path from \( x \) to \( x' \) different from \( e \), contradicting the assumption that \( \pi \in C_e \).
		
		Similarly, if \( |p_v - \mathcal{R}_{\pi}| = 1 \) for some \( v \in \mathbf{t} \), then all but one edge of \( p_v \) are present in \( J_{\pi} \), and a redundant edge could be added to complete a cycle passing through both \( x \) and \( x' \), again contradicting \( \pi \in C_e \).
		
		However, if \( |p_v- \mathcal{R}_{\pi}| \geq 2 \) for all \( v \in \mathbf{t} \), then even though redundant edges might exist, completing a cycle would change the block structure of \( \pi \), yielding a different partition. Thus, such \( \pi \) must lie in \( C_e \).
		
		To compute the cardinality of \( C_e \), we proceed by compressing the edge \( e \) and analyzing the resulting graph \( G_{\mathcal{R}}^{\{e\}} \), where \( e \) is replaced by a single vertex \( x_e \). In this graph, we must interrupt every cycle passing through \( x_e \) \emph{twice} in order to prevent alternate connections between \( x \) and \( x' \).
		
		Let \( U_{\gamma} \) be a minimal set constructed by selecting exactly two edges from each such cycle in \( G_{\mathcal{R}}^{\{e\}} \). Then,
		\[
		C_e = \bigcup_{\gamma} \Pi(X_e, j - 1, s - 1, \mathcal{R}' - U_{\gamma}),
		\]
		where \( X_e \) is the vertex set of the compressed graph and \( \mathcal{R}' \) denotes the set of atoms corresponding to the edges in \( G_{\mathcal{R}}^{\{e\}} \).
		
		Applying the Inclusion-Exclusion principle, we obtain:
		\[
		|C_e| = \sum_u \sum_{\substack{\gamma \\ d(\gamma) = u(t - 1)}} (-1)^{u+1} \bm{\pi}(X_e, j - 1, s - 1, \mathcal{R}' - U_{\gamma}).
		\]
		Therefore, since \( \Pi(X, j, s, \mathcal{R})_{xx'} = C_{xx'} \sqcup C_e \), the computations above yield the total cardinality of \( \Pi(X, j, s, \mathcal{R})_{xx'} \).
		
		Finally, to compute the cardinality of \( \Pi(X, j, s, \mathcal{R}) \), the set of partitions in which \( x \) and \( x' \) lie in \emph{different blocks}, we must interrupt all paths connecting \( x \) and \( x' \).
		
		Let \( U_{\alpha} \) be a minimal cut as defined earlier, obtained by removing exactly one edge from each such path. Then, using the Inclusion-Exclusion principle again, we obtain:
		\[
		|\Pi(X, j, s, \mathcal{R})| = \sum_u \sum_{\substack{\alpha \\ d(\alpha) = ut}} (-1)^{u+1} \bm{\pi}(X, j, s, \mathcal{R}' - U_{\alpha}).
		\]
		This completes the proof.
	\end{proof}

	\section{Potential applications}
	
Among the properties of the function $\mathfrak{N}$ introduced in Section~4, Properties (ix) and (x) warrant special attention. In particular, valuations and norms have long been known to encode key structural features of lattices~\cite{Roman}. More recently, metrics defined on lattices have found widespread use in machine learning and pattern recognition. Applications include comparing DNA sequences in bioinformatics~\cite{biology,couthino}, monitoring network performance in communication systems~\cite{networks}, and developing classification schemes in the social sciences~\cite{Metrics}. Of particular interest is the increasing attention to measures for comparing partitions~\cite{meila3,meila4,ShortestPathJyrko}, driven in part by advances in solving the mean partition problem~\cite{LeclercPartitions,strehl-ghosh}. The function $\mathfrak{N}$ and its structural constraints suggest new avenues for metric construction, and the recursive formula for $\bm{\pi}(X,j,s,\mathcal{R})$ offers a novel tool whose utility in these domains remains to be fully explored. Connections between set partitions and integer partitions may also provide combinatorial insights with potential applications in enumerative theory~\cite{Brown,Butler}.

Another key application lies in solving the \emph{mean partition problem}, in which a profile of partitions $\{\pi_1, \pi_2, \ldots, \pi_m\}$ is given and one seeks a consensus partition $\pi^{\ast}$ that minimizes a distortion function
\[
\varphi(\pi) = \sum_{i=1}^m D(\pi, \pi_i),
\]
for a suitable distance \( D \) on partitions. Various pruning criteria have been proposed to reduce the search space. In particular, the formula $\bm{\pi}(X,j,s,\mathcal{R})$ becomes essential when using the ``undesired atoms” criterion, where atoms that do not refine any member of the profile are excluded from refining the consensus $\pi^{\ast}$ \cite{correareduction}.

Finally, a promising direction emerges in the context of graph-based chemistry \cite{topologicalchemistry}. Discrete mathematical models represent molecular structures as graphs, with vertices denoting atoms and edges representing chemical bonds. Once in graph form, molecular descriptors—such as atom-bond connectivity, geometric-arithmetic indices, and distance-based indices—serve to correlate structural and experimental properties. The function $\bm{\pi}(X,j,s,\mathcal{R})$ offers a new class of topological indices, capturing structural regularities and decomposition complexity, and may provide predictive power in modeling biological and physical characteristics of chemical compounds.

\section*{Conclusion}
In this work, we investigated atomic decompositions within geometric lattices isomorphic to the lattice of partition $\Pi(X)$ of finite set $X$. Our exploration of the function $\mathfrak{N}$ described in detail the behavior of minimal atomic decompositions in $\Pi(X)$. Furthermore, our analysis of decompositions restricted a distinguished subset of atoms (i.e., the red atoms) led to a novel recursive enumeration formula for structure-constrained atomic decompositions. 

These findings are not only theoretically significant but also point to a wide range of potential applications. In particular, Properties (ix) and (x) of the function $\mathfrak{N}$ suggest connections to valuation theory and  metrics on lattices. Such metrics have practical use in diverse areas such as machine learning, bioinformatics, social sciences, and network monitoring. The recursive formula for $\bm{\pi}(X,j,s,\mathcal{R})$ offers a new tool to explore structure-constrained decompositions in these contexts. It may prove especially valuable in solving the \emph{mean partition problem}, where constraints on refining atoms allow for targeted pruning of the solution space.

Additionally, the combinatorial structure of our formula and its connection to red atoms could support advances in graph-based chemistry. Here, $\bm{\pi}(X,j,s,\mathcal{R})$ may serve as a novel topological index reflecting molecular complexity, potentially aiding the prediction of biological and physical properties in chemical compounds.

Future work may extend these results to broader classes of lattices, refine their combinatorial interpretations, or examine their algorithmic potential in data-driven applications.

\noindent\textbf{Author Contribution: }Conceptualization, Jyrko Correa-Morris; methodology, Jyrko Correa-Morris and Elvis Cabrera; software, Alex Aguila; validation, Jyrko Correa-Morris, Elvis Cabrera and Alex Aguila; formal analysis, Jyrko Correa-Morris, Elvis Cabrera and Alex Aguila; investigation, Jyrko Correa-Morris, Elvis Cabrera and Alex Aguila; resources, Jyrko Correa-Morris, Elvis Cabrera and Alex Aguila; writing---original draft preparation, Jyrko Correa-Morris and Elvis Cabrera; writing---review and editing, Jyrko Correa-Morris and Elvis Cabrera; supervision, Jyrko Correa-Morris; project administration, Jyrko Correa-Morris; funding acquisition, Jyrko Correa-Morris. All authors have read and agreed to the published version of the manuscript.

\noindent\textbf{Funding: }This research was funded by the National Science Foundation (NSF) under the STEM-UR project, Award Number 2412543, Managing Division: DUE. The APC was funded by the same grant.

\noindent\textbf{Acknowledgments:}	The authors would like to thank the STEM-UR student research team for their valuable input and discussions. During the preparation of this manuscript, the authors used ChatGPT-4 (OpenAI, 2025) for the purposes of editing assistance and language refinement. The authors have reviewed and edited the output and take full responsibility for the content of this publication.

\end{document}